\documentclass[10pt]{amsart}
\usepackage{graphicx}
\usepackage{amscd}
\usepackage{amsmath}
\usepackage{amsthm}
\usepackage{amsfonts}
\usepackage{amssymb}
\usepackage{mathrsfs}
\usepackage{enumitem}
\usepackage{amsrefs}
\usepackage{xcolor}
\usepackage[colorlinks, citecolor=blue, linkcolor=red, pdfstartview=FitB]{hyperref}
\usepackage[latin1]{inputenc}
\usepackage{mathdots}

\newcommand{\bB}{{\mathbb{B}}}
\newcommand{\bC}{{\mathbb{C}}}
\newcommand{\bD}{{\mathbb{D}}}

\newcommand{\bF}{{\mathbb{F}}}

\newcommand{\bM}{{\mathbb{M}}}
\newcommand{\bN}{{\mathbb{N}}}

\newcommand{\bT}{{\mathbb{T}}}

  \newcommand{\A}{{\mathcal{A}}}
  \newcommand{\B}{{\mathcal{B}}}
  
  \newcommand{\D}{{\mathcal{D}}}
  
  \newcommand{\F}{{\mathcal{F}}}
  
\renewcommand{\H}{{\mathcal{H}}}
  \newcommand{\I}{{\mathcal{I}}}
  \newcommand{\J}{{\mathcal{J}}}
    
\renewcommand{\L}{{\mathcal{L}}}
  \newcommand{\M}{{\mathcal{M}}}

\renewcommand{\S}{{\mathcal{S}}}
  \newcommand{\T}{{\mathcal{T}}}

  \newcommand{\X}{{\mathcal{X}}}

\newcommand{\fA}{{\mathfrak{A}}}
\newcommand{\fB}{{\mathfrak{B}}}

\newcommand{\fF}{{\mathfrak{F}}}

\newcommand{\fH}{{\mathfrak{H}}}

\newcommand{\fK}{{\mathfrak{K}}}
\newcommand{\fL}{{\mathfrak{L}}}
\newcommand{\fM}{{\mathfrak{M}}}
\newcommand{\fN}{{\mathfrak{N}}}

\newcommand{\fT}{{\mathfrak{T}}}

\newcommand{\rC}{\mathrm{C}}

\newcommand{\eps}{\varepsilon}
\renewcommand{\phi}{\varphi}

\newcommand{\upchi}{{\raise.35ex\hbox{$\chi$}}}

\newcommand{\ol}{\overline}

\newcommand{\qand}{\quad\text{and}\quad}

\newcommand{\id}{\operatorname{id}}

\newcommand{\spn}{\operatorname{span}}

\newtheorem{lemma}{Lemma}[section]
\newtheorem{theorem}[lemma]{Theorem}
\newtheorem{proposition}[lemma]{Proposition}
\newtheorem{corollary}[lemma]{Corollary}
\theoremstyle{definition}

\newtheorem{example}{Example}

\begin{document}
\author{Rapha\"el Clou\^atre}

\address{Department of Mathematics, University of Manitoba, Winnipeg, Manitoba, Canada R3T 2N2}

\email{raphael.clouatre@umanitoba.ca\vspace{-2ex}}
\thanks{The first author was partially supported by an NSERC Discovery Grant}
\author{Christopher Ramsey}
\email{christopher.ramsey@umanitoba.ca\vspace{-2ex}}

\begin{abstract}
We study non-selfadjoint operator algebras that can be entirely understood via their finite-dimensional representations. In contrast with the elementary matricial description of finite-dimensional $\rC^*$-algebras, in the non-selfadjoint setting we show that an additional level of flexibility must be allowed. Motivated by this peculiarity, we consider a natural non-selfadjoint notion of residual finite-dimensionality. We identify sufficient conditions for the tensor algebra of a $\rC^*$-correspondence to enjoy this property. To clarify the connection with the usual self-adjoint notion, we investigate the residual finite-dimensionality of  the minimal and maximal $\rC^*$-covers associated to an operator algebra.

\end{abstract}

\title{Residually finite-dimensional operator algebras}
\date{\today}
\maketitle

\section{Introduction}\label{S:intro}

Finite-dimensional $\rC^*$-algebras are easily understood as direct sums of matrix algebras. In trying to understand arbitrary $\rC^*$-algebras, it is therefore natural to approximate them, whenever possible, with finite-dimensional ones. This general strategy has led to the introduction of various important properties of $\rC^*$-algebras, such as nuclearity and quasidiagonality (see \cite{BO2008} for a detailed account).  In view of the spectacular recent progress in the structure theory of $\rC^*$-algebras based on the idea of finite-dimensional approximations (see for instance \cite{TWW2017}), one may want to proceed along similar lines to clarify the structure of \emph{non-selfadjoint} operator algebras, and such is the motivation for this paper.

Perhaps the most basic finite-dimensional approximation property that a $\rC^*$-algebra can enjoy is that of residual finite-dimensionality. The class of residually finite-dimensional (RFD) $\rC^*$-algebras consists of those that can be embedded in a product of matrix algebras. In other words, RFD $\rC^*$-algebras admit block-diagonal decompositions with finite-dimensional blocks. This class contains the familiar commutative $\rC^*$-algebras, but also some more complicated objects; a classical example is the full 
$\rC^*$-algebra of the free group on two generators \cite{choi1980}. Furthermore, any $\rC^*$-algebra is a quotient of an RFD $\rC^*$-algebra \cite{GM1990}. Throughout the years, several characterizations of RFD $\rC^*$-algebras have emerged \cite{EL1992},\cite{archbold1995},\cite{hadwin2014},\cite{courtney2017}. Studying this property in the setting of non-selfadjoint operator algebras is the driving force of this paper. Similar investigations can be found scattered in the literature (see for instance \cite{mittal2010} and \cite{CM2017rfd}), but we adopt here a somehow more systematic approach. 

Already, the mere definition of what it should mean for a general operator algebra $\A$ to be RFD raises interesting questions. Assume for instance that $\A$ can be approximated, in some sense, by finite-dimensional operator algebras. It is not obvious at first glance whether the finite-dimensional approximating algebras can be taken to be comprised of matrices. The point here is that the structure of finite-dimensional operator algebras is not as transparent as that of their self-adjoint counterparts. Clarifying this issue is one our objectives. 

The second main goal of the paper is to relate and contrast residual finite dimensionality in the self-adjoint world with the corresponding property in the non-selfadjoint world. We approach this question by starting with an RFD operator algebra $\A$, and investigating whether the $\rC^*$-algebras that various copies of $\A$ generate (the so-called \emph{$\rC^*$-covers} of $\A$) are also RFD. In fact, we will focus on two particularly important $\rC^*$-covers: the maximal $\rC^*$-cover $\rC^*_{\max}(\A)$, and the minimal one $\rC^*_e(\A)$, which is typically called the $\rC^*$-envelope. We now describe the organization of the paper, and state our main results.

Section \ref{S:prelim} introduces some necessary background material.

In Section \ref{S:structure}, we perform a careful analysis of finite-dimensional operator algebras. As opposed to the self-adjoint setting, finite-dimensional non-selfadjoint operator algebras may not be classified up to completely isometric isomorphism using matrix algebras. Such a simple description is available if one is willing to settle for a classification up to completely bounded isomorphisms (Proposition \ref{P:fdimoacb} and Corollary \ref{C:fdoastructure}). 
The main results of the section (Theorems \ref{T:fdimRFDnorm} and \ref{T:RFDbimodule}) show that finite-dimensional operator algebras can be well-approximated by matrix algebras.

The information about finite-dimensional operator algebras obtained in Section \ref{S:structure} is leveraged in Section \ref{S:RFD}, where we turn to the study of residually finite-dimensional non-selfadjoint operator algebras. The main result of the section is the following (Theorem \ref{T:fdimvsfdim}), which shows that despite the lack of a completely isometric classification of finite-dimensional operator algebras using matrix algebras, the two classes can be used interchangeably in the definition of an RFD operator algebra. This is consistent with the self-adjoint setting.

\begin{theorem}\label{T:mainA}
Let $\A$ be an operator algebra. Consider the following statements.
\begin{enumerate}

\item[\rm{(i)}] There is a collection $\{\B_i\}_{i\in \Omega}$ of finite-dimensional operator algebras and a completely isometric homomorphism $\Phi:\A\to \prod_{i\in \Omega}\B_i$.

\item[\rm{(ii)}] There is a collection $\{\fB_i\}_{i\in \Omega}$ of finite-dimensional $\rC^*$-algebras and a completely isometric homomorphism $\Phi:\A\to \prod_{i\in \Omega}\fB_i$.

\item[\rm{(iii)}] For every $d\in\bN$ and every $A\in \bM_d(\A)$, there is a finite-dimensional operator algebra $\B$ and a completely contractive homomorphism  $\pi:\A\to\B$ such that $\|\pi^{(d)}(A)\|=\|A\|$.

\item[\rm{(iv)}] For every $d\in\bN$ and every $A\in \bM_d(\A)$, there is a finite-dimensional $\rC^*$-algebra $\fB$ and a completely contractive homomorphism  $\pi:\A\to\fB$ such that $\|\pi^{(d)}(A)\|=\|A\|$.
\end{enumerate}

Then, \emph{(i)} and \emph{(ii)} are equivalent, and \emph{(iii)} and \emph{(iv)} are equivalent.
\end{theorem}

We pay close attention to a very important class of operator algebras, namely the tensor algebras of $\rC^*$-correspondences. We study them carefully through the lens of residual finite-dimensionality. For instance, we obtain the following (Theorems \ref{T:fdcorrespondence} and \ref{T:graph}).

\begin{theorem}\label{T:mainE}
The following statements hold.

\begin{enumerate}
\item[\rm{(1)}]
 Let $\fA$ be a finite-dimensional $\rC^*$-algebra and let $X$ be a $\rC^*$-corres\-pondence over $\fA$. Then, the tensor algebra $\T_{X}^+$ is RFD.

\item[\rm{(2)}] Let $G$ be a directed graph  and let $X_G$ be the associated graph correspondence. Then,  the tensor algebra $\T^+_{X_G}$ is RFD.
\end{enumerate}

\end{theorem}

For the remaining two sections, our focus shifts from residual finite-dimensionality of non-selfadjoint operator algebras to that of some of their $\rC^*$-covers. First, in Section \ref{S:C*max}, we consider the maximal $\rC^*$-cover. The situation is particularly transparent for finite-dimensional operator algebras (Theorem \ref{T:C*maxfdimA}).

\begin{theorem}\label{T:mainB}
Let $\A$ be a finite-dimensional operator algebra. Then, $\rC^*_{\max}(\A)$ is RFD.
\end{theorem}

We then ask whether $\A$ being RFD is equivalent to $\rC^*_{\max}(\A)$ being RFD, and exhibit supporting examples and sufficient conditions for that equivalence to hold (Corollary  \ref{C:RFDdirectsum} and Theorem \ref{T:idealRFD}).

Finally, in Section \ref{S:C*env} we replace the maximal $\rC^*$-cover in the previous considerations by the minimal one, also known as the $\rC^*$-envelope. We exhibit an example of an RFD operator algebra $\A$ for which $\rC^*_e(\A)$ is not RFD. The main results of the section  identify conditions on an RFD operator algebra $\A$ that are sufficient for the $\rC^*$-envelope $\rC^*_e(\A)$ to be RFD (Theorems \ref{T:quotientRFD} and \ref{T:epssurj}). To state these results, we need the following notation. Let $(r_n)_n$ be a sequence of positive integers. For each $m\in \bN$ we let $\gamma_m:\prod_{n=1}^\infty \bM_{r_n}\to \bM_{r_m}$ denote the natural projection. Let $\A\subset \prod_{n=1}^\infty \bM_{r_n}$ be a unital operator algebra and let $\fK$ denote the ideal of compact operators in $\rC^*(\A)$.

\begin{theorem}\label{T:mainD}
The following statements hold. 
\begin{enumerate}
\item[\rm{(1)}]  Assume that every $\rC^*$-algebra which is a quotient of $\rC^*(\A)/\fK$ is RFD. Then, $\rC^*_e(\A)$ is RFD. In particular, this holds if $\rC^*(\A)/\fK$ is commutative or finite-dimensional.

\item[\rm{(2)}]  Assume that there is $N\in \bN$ with the property that $\gamma_n|_{\fK\cap \A}$ is a complete quotient map onto $\bM_{r_n}$ for every $n\geq N$. Then, $\rC^*_e(\A)$ is RFD.
\end{enumerate}
\end{theorem}

\textbf{Acknowledgements.} The first author wishes to thank Matt Kennedy for a stimulating discussion which brought \cite{pestov1994} to his attention and sparked his interest in the residual finite-dimensionality of $\rC^*$-covers.

\section{Preliminaries}\label{S:prelim}

\subsection{Operator algebras and $\rC^*$-covers}
Throughout the paper, $\fH$ will denote a complex Hilbert space and $B(\fH)$ will denote the space of bounded linear operators on it.  An \emph{operator algebra} is simply a norm closed subalgebra $\A\subset B(\fH)$. It will be said to be unital if it contains the identity on $\fH$. Given a positive integer $n\in \bN$, we denote by $\bM_n(\A)$ the space of $n\times n$ matrices with entries in $\A$. When $\A=\bC$, we simply write $\bM_n$ instead of $\bM_n(\bC)$. The norm on $\bM_n(\A)$ is that inherited from $B(\fH^{(n)})$, where $\fH^{(n)}=\fH\oplus \fH\oplus \ldots \oplus \fH$. Given a linear map $\phi:\A\to B(\fH_\phi)$, we denote by $\phi^{(n)}$ the natural ampliation to $\bM_n(\A)$. Recall that $\phi$ is said to be \emph{completely contractive} (respectively,  \emph{completely isometric}) if $\phi^{(n)}$ is contractive (respectively, isometric) for every $n\in \bN$. More generally, $\phi$ is \emph{completely bounded} if the quantity
\[
\|\phi\|_{cb}=\sup_{n\in \bN}\|\phi^{(n)}\|
\]
is finite. The reader may consult \cite{paulsen2002} for details.

Typically, we consider an operator algebra $\A$ to be determined only up to completely isometric isomorphism. In particular, there are many different $\rC^*$-algebras that a copy of $\A$ can generate, and the following notion formalizes this idea. A \emph{$\rC^*$-cover} of $\A$ is a pair $(\fA, \iota)$ consisting of a $\rC^*$-algebra $\fA$ and a complete isometric homomorphism $\iota : \A \rightarrow \fA$ such that $\rC^*(\iota(\A)) = \fA$. For our purposes, we will be focusing on two particular $\rC^*$-covers, which we now describe.

The \emph{maximal }$\rC^*$-cover $(\rC^*_{\max}(\A),\mu)$ of $\A$ is the essentially unique $\rC^*$-cover with the property that whenever $\phi:\A\to B(\fH_\phi)$ is a completely contractive homomorphism, there is a $*$-homomorphism $\pi_\phi:\rC^*_{\max}(\A)\to B(\fH_\phi)$ with the property that $\pi_\phi\circ \mu=\phi$ on $\A$.  The algebra $\rC_{\max}^*(\A)$ can be realized as the $\rC^*$-algebra generated by the image of $\A$ under an appropriate direct sum of completely contractive homomorphisms \cite{blecher1999}. 

There is a purely linear version of this construction which we will require as well.  Let $\M\subset B(\fH)$ be a subspace. We can associate to it a ``free" $\rC^*$-algebra $\rC^*\langle \M \rangle$ and a completely isometric linear map $\mu:\M\to \rC^*\langle \M \rangle $ such that $\rC^*\langle \M \rangle=\rC^*(\mu(\M))$, and whenever $\phi:\M\to B(\fH_\phi)$ is a completely contractive linear map, there is a unital $*$-homomorphism $\pi_\phi:\rC^*\langle \M \rangle\to B(\fH_\phi)$ with the property that $\pi_\phi\circ \mu=\phi$ on $\M$. Once again, $\rC^*\langle \M \rangle$ can be realized more concretely as the $\rC^*$-algebra generated by the image of $\M$ under an appropriate direct sum of completely contractive linear maps \cite[Theorem 3.2]{pestov1994}.

We now turn to the ``minimal" $\rC^*$-cover, which is the so-called $\rC^*$-envelope of a unital operator algebra. In fact, it will be convenient for us to give the definition for general unital subspaces $\S\subset B(\fH)$ rather than operator algebras. Let $\eps$ be a unital completely isometric linear map on $\S$. Then, $\rC^*(\eps(\S))$ is the \emph{$\rC^*$-envelope} of $\S$, denoted by $\rC^*_e(\S)$, if whenever $\phi:\S\to B(\fH_\phi)$ is a unital completely isometric linear map, there is a $*$-homomorphism $\pi:\rC^*(\phi(\S))\to \rC^*_e(\S)$ with the property that $\pi \circ \phi=\eps$ on $\S$. The uniqueness of such an object is easily verified, but the existence of the $\rC^*$-envelope is non-trivial, and it was first established in \cite{hamana1979}. Practically speaking, an approach pioneered by Arveson \cite{arveson1969} is often more useful to identify the $\rC^*$-envelope. This approach is based on a rather deep analogy with the classical theory of uniform algebras and the Shilov and Choquet boundaries. We recall the details that will be relevant for us. 

Assume that $\S\subset B(\fH)$ is a unital subspace. A unital completely contractive linear map $\phi:\S\to B(\fH_\phi)$ always admits a unital completely contractive extension to $\rC^*(\S)\subset B(\fH)$ by Arveson's extension theorem. Accordingly, we say that a unital $*$-homomorphism $\pi:\rC^*(\S)\to B(\fH_\pi)$ has the \emph{unique extension property} with respect to $\S$ if it is the only unital completely contractive extension to $\rC^*(\S)$ of $\pi|_{\S}$. It is known \cite{arveson1969} that if a unital $*$-homomorphism has the unique extension property with respect to $\S$ and $\pi|_{\S}$ is completely isometric, then we can choose $\eps=\pi|_\S$ and thus $\rC^*_e(\S)\cong \pi(\rC^*(\S))$.  In this case, $\ker \pi$ is called the \emph{Shilov ideal} of $\S$. We note that $\rC^*_e(\S)\cong \rC^*(\S)/\ker \pi$, and the defining property of the $\rC^*$-envelope implies that the Shilov ideal is the largest closed two-sided ideal $\J$ of $\rC^*(\S)$ with the property that the quotient map $\rC^*(\S)\to \rC^*(\S)/\J$ is completely isometric on $\S$.
 Finally, we emphasize there are known mechanisms to produce such $*$-homomorphisms with the unique extension property with respect to $\S$ which are completely isometric on $\S$ \cite{dritschel2005},\cite{arveson2008},\cite{davidsonkennedy2015}.

\subsection{Residual finite dimensionality}

Let $\fA$ be a $\rC^*$-algebra.  Then, $\fA$ is said to be \emph{residually finite-dimensional} (henceforth abbreviated to RFD) if it admits a separating family of finite-dimensional $*$-representations. In other words, $\fA$ is RFD if there is a set of positive integers $\{r_\lambda:\lambda\in \Lambda\}$ and an injective $*$-homomorphism 
\[
\pi:\fA\to \prod_{\lambda\in \Lambda}\bM_{r_\lambda}.
\]
Equivalently, the map $\pi$ must be completely isometric.

As done in \cite{CM2017rfd}, we can extend this definition to general operator algebras. An operator algebra $\A$ is RFD  if there is a set of positive integers $\{r_\lambda:\lambda\in \Lambda\}$ and a completely isometric homomorphism
\[
\rho:\A\to \prod_{\lambda\in \Lambda}\bM_{r_\lambda}.
\]
Upon recalling that a completely contractive homomorphism on a $\rC^*$-algebra is necessarily positive, we see that this definition agrees with the previous one whenever $\A$ happens to be self-adjoint.

\section{Structure of finite-dimensional operator algebras}\label{S:structure}

\subsection{Completely bounded embeddings in matrix algebras}

Before proceeding with our investigation of RFD operator algebras, we first need to understand finite-dimensional ones. By analogy with finite-dimensional $\rC^*$-algebras, one may naively conjecture that finite-dimensional operator algebras are exactly those which are completely isometrically isomorphic to subalgebras of direct sums of matrix algebras. In the unital case, this is equivalent to admitting a finite-dimensional $\rC^*$-envelope. This conjecture is supported by  \cite[Theorem 4.2]{meyer2001} in the case of unital two-dimensional operator algebras. However, typically things are not so straightforward. Before illustrating this fact with an example, we record a useful calculation that will be used several times throughout.

\begin{lemma}\label{L:M2C*env}
Let $\S\subset B(\fH)$ be a unital subspace and let $\A_{\S}\subset B(\fH^{(2)})$ be the unital operator algebra consisting of elements of the form
$
\begin{bmatrix}
\lambda I & s\\
0 & \mu I
\end{bmatrix}
$
for some $s\in \S$ and $\lambda,\mu\in \bC$. Then, we have that
\[
\rC^*_e(\A_{\S})\cong\bM_2(\rC_e^*(\S)).
\]
\end{lemma}
\begin{proof}
A straightforward calculation shows that 
\[
\rC^*(\A_{\S})=\bM_2(\rC^*(\S)).
\]
Moreover, a routine argument using matrix units reveals that $\J\subset\bM_2(\rC^*(\S))$ is a closed two-sided ideal if and only if there is a closed two-sided ideal $\I\subset \rC^*(\S)$ with the property that $\J=\bM_2(\I)$. Therefore, if we let $\Sigma\subset \rC^*(\S)$ denote the Shilov ideal of $\S$, then we find that $\bM_2(\Sigma)$ is the Shilov ideal of $\A_{\S}$. Hence,
\begin{align*}
\rC^*_e(\A_{\S})&\cong \rC^*(\A_{\S})/\bM_2(\Sigma)=\bM_2(\rC^*(\S))/\bM_2(\Sigma)\\
&\cong \bM_2(\rC^*(\S)/\Sigma)\cong \bM_2(\rC_e^*(\S)).
\end{align*}
\end{proof}

Using this fact, we can give an example of a finite-dimensional unital operator algebra with infinite-dimensional $\rC^*$-envelope. The following is \cite[Exercise 15.12]{paulsen2002}.

\begin{example}\label{E:fdoa}
Let $\rC(\bT)$ denote the unital $\rC^*$-algebra of continuous functions on the unit circle $\bT$ and consider the unital subspace $\S=\spn\{1,z\}\subset \rC(\bT)$. For each $\zeta\in \bT$, we define the function $\phi_\zeta\in \S$ as
\[
\phi_\zeta(z)=\frac{1}{2}(1+\ol{\zeta}z), \quad z\in \bT.
\]
Then, $\phi_\zeta$ peaks at $\zeta$, which forces $\zeta$ to belong to the Shilov boundary of $\S$ (see \cite{phelps2001} for details). Thus, the Shilov boundary of $\S$ is $\bT$ and thus $\rC^*_e(\S)\cong \rC(\bT)$.

Now, let $\A_\S\subset \bM_2(\rC(\bT))$ be the unital operator algebra defined in Lemma \ref{L:M2C*env}. Then, we see that $\A_\S$ is finite-dimensional and
\[
\rC_e^*(\A_\S)\cong \bM_2(\rC^*_e(\S))\cong \bM_2(\rC(\bT))
\]
is infinite-dimensional.
 \qed
\end{example}

To reiterate, the operator algebra $\A_{\S}$ in the example above cannot be embedded completely isometrically isomorphically in a matrix algebra, for then $\rC_e^*(\A_\S)$ would be finite-dimensional. Hence, finite-dimensional operator algebras exhibit more varied behaviour than their self-adjoint counterparts. Nevertheless, we note that the classical Artin-Wedderburn theorem can be used to show that semisimple finite-dimensional operator algebras are \emph{similar} to direct sums of matrix algebras. It thus appears that if we are willing to settle for a softer classification of finite-dimensional operator algebras, replacing completely isometric isomorphisms by merely completely bounded ones, then we can recover the familiar description available for $\rC^*$-algebras. This is indeed the case, and establishing this fact is the first goal of this section. One of the basic ingredients is the following.

\begin{proposition}\label{P:fdimoacb}
Let $\A\subset B(\fH)$ be an operator algebra with dimension $d$. Then, there is a positive integer $r\geq 1$, a subalgebra $\B\subset \bM_r$ and a completely contractive algebra isomorphism $\Phi:\A\to \B$ with  $\|\Phi^{-1}\|_{cb} \leq 2d.$ If $\A$ is unital, then $\Phi$ can be chosen to be unital.
\end{proposition}
\begin{proof}
By \cite[Proposition 5.3]{CM2017rfd}, there is a positive integer $r\in \bN$, a unital subalgebra $\B\subset \bM_r$ and a unital completely contractive isomorphism $\Phi:\A\to \B$ with the property that $\|\Phi^{-1}\|\leq 2$. Thus, \cite[Proposition 2.8]{paulsen1992} implies that $\|\Phi^{-1}\|_{cb}\leq 2d.$ In the unital case, inspection of the proof of  \cite[Proposition 5.3]{CM2017rfd} reveals that $\Phi$ can be chosen to be unital.
\end{proof}

We now describe the other ingredient that we require. Given two operator algebras $\A\subset B(\fH_1)$ and $\B\subset B(\fH_2)$, an isomorphism $\Phi:\A\to\B$ will be called a \emph{completely bounded isomorphism} if $\Phi$ and $\Phi^{-1}$ are completely bounded.
A classical theorem of Paulsen \cite{paulsen1984}, \cite{paulsen1984PAMS} says that completely bounded homomorphisms on operator algebras are necessarily similar to completely contractive ones. In \cite{clouatre2015CB}, the possibility of obtaining a ``two-sided" version of Paulsen's theorem for completely bounded isomorphisms was investigated. More precisely,  the question is this: given a completely bounded isomorphism $\Phi:\A\to\B$, do there exist two invertible operators $X\in B(\fH_1)$ and $Y\in B(\fH_2)$ such that the map
\[
XaX^{-1}\mapsto Y\Phi(a)Y^{-1}, \quad a\in \A
\]
is a complete isometry? It was shown in \cite{clouatre2015CB} that in general the answer is no. We show next that a weaker statement always hold.

\begin{theorem}\label{T:rigidity}
Let $\A\subset B(\fH_1),\B\subset B(\fH_2)$ be unital operator algebras and let $\Phi:\A\to \B$ be a unital completely bounded isomorphism. Then, there are two unital completely isometric homomorphisms 
\[
\lambda:\A\to B(\fH_1)\oplus B(\fH_2), \quad \rho:\B\to B(\fH_1)\oplus B(\fH_2)
\]
along with an invertible operator $Z\in B(\fH_1)\oplus B(\fH_2)$ with the property that
\[
\Phi(a)=\rho^{-1}(Z\lambda(a)Z^{-1}), \quad a\in \A.
\]
\end{theorem}
\begin{proof}
By \cite[Theorem 3.1]{paulsen1984}, there exist invertible operators $X\in B(\fH_1)$ and $Y\in B(\fH_2)$ such that the maps
\[
a\mapsto Y\Phi(a)Y^{-1}, \quad a\in \A
\]
\[
b\mapsto X\Phi^{-1}(b)X^{-1}, \quad b\in \B
\]
are completely contractive. Define
\[
\lambda:\A\to B(\fH_1)\oplus B(\fH_2)
\]
as
\[
\lambda(a)=a\oplus Y\Phi(a)Y^{-1}, \quad a\in \A
\] 
and
\[
\rho:\B\to B(\fH_1)\oplus B(\fH_2)
\]
as
\[
\rho(b)=X\Phi^{-1}(b)X^{-1}\oplus b, \quad b\in \B.
\] 
Then, $\lambda$ and $\rho$ are completely isometric.
Put $Z=X\oplus Y^{-1}$. Then,
\begin{align*}
Z\lambda(a)Z^{-1}&=XaX^{-1}\oplus \Phi(a)=\rho(\Phi(a))
\end{align*}
for every $a\in \A$.
\end{proof}

Next, we use the previous result to achieve our first goal and further elucidate the structure of finite-dimensional operator algebras. Roughly speaking, we show that up to a similarity, finite-dimensional unital operator algebras admit finite-dimensional $\rC^*$-envelopes.

\begin{corollary}\label{C:fdoastructure}
Let $\A\subset B(\fH)$ be a unital operator algebra. Then, $\A$ is finite-dimensional if and only if there is another unital operator algebra $\F$ that is completely isometrically isomorphic to $\A$ and that is similar to an operator algebra whose $\rC^*$-envelope is finite-dimensional.
\end{corollary}
\begin{proof}
It is clear that $\A$ is finite-dimensional if there exists an algebra $\F$ with the announced properties. Conversely, assume that $\A$ is finite-dimensional. By Proposition \ref{P:fdimoacb}, there is a positive integer $r\geq 1$, a unital subalgebra $\B\subset \bM_r$ and a unital completely bounded isomorphism $\Phi:\A\to \B$. Next, apply Theorem \ref{T:rigidity} to the map $\Phi$ and find two unital completely isometric homomorphisms 
\[
\lambda:\A\to B(\fH)\oplus \bM_r, \quad \rho:\B\to B(\fH)\oplus\bM_r
\]
along with an invertible operator $Z\in B(\fH)\oplus \bM_r$ with the property that
\[
\Phi(a)=\rho^{-1}(Z\lambda(a)Z^{-1}), \quad a\in \A.
\]
We note that $\B\subset \bM_r$, so that $\rC^*_e(\B)$ is finite-dimensional. Since $\rho$ is a unital completely isometric homomorphism, we see that $\rC^*_e(\rho(\B))\cong \rC^*_e(\B)$ is finite-dimensional as well. Finally, we put $\F=Z^{-1}\rho(\B)Z$ and note that
\[
\F=Z^{-1}\rho(\B)Z=Z^{-1}\rho(\Phi(\A))Z=\lambda(\A)
\]
so that indeed $\F$ is completely isometrically isomorphic to $\A$.
\end{proof}

\subsection{Residual finite dimensionality}

Next, we proceed to show that finite-\\dimensional operator algebras are RFD. Notice that in view of Example \ref{E:fdoa}, this is not immediate unlike in the self-adjoint setting. In fact, we obtain more precise information.

\begin{theorem}\label{T:fdimRFDnorm}
Let $\A\subset B(\fH)$ be a finite-dimensional operator algebra, let $d\in \bN$ and let $A\in \bM_d(\A)$. Then, there is a finite-dimensional Hilbert space $\fF$ and a completely contractive homomorphism $\pi:\A\to B(\fF)$ with the property that $\|\pi^{(d)}(A)\|=\|A\|$.
\end{theorem}
\begin{proof}
Let $\fA=\rC^*(\A)+\bC I_\fH\subset B(\fH)$. There is a state $\psi$ of $\bM_d(\fA)$ with the property that $\psi(A^*A)=\|A\|^2$. Let $\sigma_\psi:\bM_d(\fA)\to B(\fH_\psi)$ be the associated GNS representation, with unit cyclic vector $\xi_\psi$. Then
\[
\|\sigma_\psi(A)\xi_\psi\|^2=\langle \sigma_\psi(A^*A)\xi_\psi,\xi_\psi \rangle=\psi(A^*A)=\|A\|^2.
\]
Now, it is well known (see for instance \cite{hopenwasser1973}) that there is a Hilbert space $\fH'$, a unitary operator $U:\fH_\psi\to \fH'^{(d)}$ and a unital $*$-homomorphism $\tau: \fA\to B(\fH')$ such that
\[
U\sigma_\psi(B)U^*=\tau^{(d)}(B), \quad B\in \bM_d(\fA).
\]
Write $U\xi_\psi=\xi_1\oplus \ldots\oplus \xi_d$ for some $\xi_1,\ldots,\xi_d\in \fH'$. Let $\fF\subset \fH'$ be the subspace spanned by $\xi_1,\ldots,\xi_d$ and $\tau(\A)\xi_1,\ldots, \tau(\A)\xi_d$. Then, $\fF$ is clearly invariant for $\tau(\A)$ and finite-dimensional. The associated restriction $\pi:\A\to B(\fF)$ defined as
\[
\pi(b)=\tau(b)|_{\fF}, \quad b\in \A
\]
is a completely contractive homomorphism that satisfies
\[
\|\pi^{(d)}(A)\|\geq \|\tau^{(d)}(A)U\xi_\psi\|=\|U\sigma_\psi(A)\xi_\psi\|=\|A\|.
\]
\end{proof}

We now obtain the announced residual finite-dimensionality result.

\begin{corollary}\label{C:fdimRFD}
Finite-dimensional operator algebras are RFD.
\end{corollary}
\begin{proof}
This follows immediately from Theorem \ref{T:fdimRFDnorm}.
\end{proof}

In \cite{courtney2017}, the authors explore the connection between the residual finite-dimensionality  of a $\rC^*$-algebra and the abundance of elements in it that attain their norms in finite-dimensional representations. Beyond Theorem \ref{T:fdimRFDnorm}, we do not know whether analogous results hold in the non-selfadjoint context.

We close this section by refining Corollary \ref{C:fdimRFD}.

\begin{theorem}\label{T:RFDbimodule}
Let $\A\subset B(\fH)$ be a finite-dimensional operator algebra. Then, there is a set of positive integers $\{r_\lambda:\lambda\in \Lambda\}$ and a unital completely isometric map
\[
\Psi:B(\fH)\to \prod_{\lambda\in \Lambda}\bM_{r_\lambda}
\]
such that
\[
\Psi(a^* t b)=\Psi(a)^* \Psi(t) \Psi(b)
\]
for every $a,b\in \A$ and $t\in B(\fH)$. Moreover, $\Psi$ restricts to a homomorphism on $\A$.
\end{theorem}
\begin{proof}
Let $\Xi\subset \fH$ be a finite set of vectors. Let $\X_\Xi\subset \fH$ be the subspace spanned by $\xi$ and $\A \xi$ for every $\xi\in \Xi$. Then, $\X_\Xi$ is finite-dimensional and we  put $n_\Xi=\dim \X_\xi$. Upon identifying $B(\X_{\Xi})$ with $\bM_{n_{\Xi}}$, we may define a map
\[
\rho_{\Xi}:B(\fH)\to \bM_{n_{\Xi}}
\]
as
\[
\rho_{\Xi}(t)=P_{\X_{\Xi}}t|_{\X_{\Xi}}, \quad t\in B(\fH).
\]
It is immediate that $\rho_{\Xi}$ is unital and completely contractive, and in particular it is self-adjoint. Furthermore, it is clear from its definition that the subspace $\X_{\Xi}$ is invariant for $\A$. Thus, we have that $\rho_\Xi$ restricts to a homomorphism on $\A$ and
\begin{align*}
\rho_{\Xi}(a^* tb)&=P_{\X_{\Xi}}a^*tb|_{\X_{\Xi}}=P_{\X_{\Xi}}a^*P_{\X_{\Xi}}tP_{\X_{\Xi}}b|_{\X_{\Xi}}\\
&=\rho_{\Xi}(a)^*\rho_{\Xi}(t)\rho_{\Xi}(b)
\end{align*}
for every $a,b\in \A$ and $t\in B(\fH)$. Define now $\Psi=\oplus_{\Xi}\: \rho_{\Xi}$,
where the direct sum extends over all finite subsets of vectors $\Xi\subset \fH$. Clearly, $\Psi$ is a unital completely contractive (and thus self-adjoint) map that restricts to a homomorphism on $\A$. Moreover
\[
\Psi(a^* t b)=\Psi(a)^* \Psi(t) \Psi(b)
\]
for every $a,b\in \A$ and $t\in B(\fH)$. It remains to show that it is completely isometric. 
To see this, fix $T=[t_{ij}]_{i,j}\in \bM_d(B(\fH))$. Let $\zeta=\zeta_1\oplus \ldots\oplus \zeta_d\in \fH^{(d)}$ be a unit vector and let
\[
\Xi=\{\zeta_j:1\leq j\leq d\}\cup \{t_{ij}\zeta_j:1\leq i,j\leq d\}\subset \fH.
\]
We have $\zeta\in\X_{\Xi}^{(d)}$ and
\[
T \zeta=\left( \sum_{j=1}^d t_{1,j}\zeta_j\right)\oplus \ldots\oplus\left( \sum_{j=1}^d t_{d,j}\zeta_j\right)\in \X_{\Xi}^{(d)}.
\]
Now, we observe that
\[
\rho_{\Xi}^{(d)}(T)=P_{ \X_{\Xi}^{(d)}}T|_{ \X_{\Xi}^{(d)}}
\]
whence
\[
\|\Psi^{(d)}(T)\|\geq  \|P_{ \X_{\Xi}^{(d)}}T|_{ \X_{\Xi}^{(d)}}\|\geq \|T\zeta\|.
\]
Since $\zeta\in \fH^{(d)}$ is an arbitrary unit vector, we obtain that $\|\Psi^{(d)}(T)\|\geq \|T\|$ so that indeed $\Psi$ is completely isometric.
\end{proof}

We obtain a curious consequence, which is likely known. It is reminiscent of \cite[Corollary 8]{choi1980}. Recall that an operator $t\in B(\fH)$ is said to be \emph{hyponormal} if $tt^*\leq t^*t$. 

\begin{corollary}\label{C:hyponormal}
Let $\A\subset B(\fH)$ be a finite-dimensional  operator algebra. Then, every hyponormal element of $\A$ is normal.
\end{corollary}
\begin{proof}
Let $a\in \A$ be hyponormal, so that $aa^*\leq a^*a$. By Theorem \ref{T:RFDbimodule}
there is a set of positive integers $\{r_\lambda:\lambda\in \Lambda\}$ and a unital completely isometric map
\[
\Psi:B(\fH)\to \prod_{\lambda\in \Lambda}\bM_{r_\lambda}
\]
such that
\[
\Psi(a^* t b)=\Psi(a)^* \Psi(t) \Psi(b)
\]
for every $a,b\in \A$ and $t\in B(\fH)$. It suffices to show that $\Psi(a^*a)=\Psi(aa^*)$. Since $\Psi$ must necessarily be completely positive, we may invoke the Schwarz inequality to find
\begin{align*}
\Psi(a)\Psi(a)^*\leq \Psi(aa^*)\leq \Psi(a^*a)=\Psi(a)^*\Psi(a).
\end{align*}
Write $\Psi(a)=(b_\lambda)_{\lambda\in \Lambda}$ where $b_\lambda\in \bM_{r_\lambda}$ for each $\lambda\in \Lambda$. Note then that
\[
b_\lambda b_\lambda^*\leq b_\lambda^* b_\lambda, \quad \lambda\in \Lambda.
\]
This implies that $b_\lambda^*b_\lambda - b_\lambda b_\lambda^*$ is a non-negative matrix with zero trace, whence $b_\lambda^*b_\lambda= b_\lambda b_\lambda^*$ for every $\lambda\in \Lambda$. In turn, this means that $\Psi(a)\Psi(a)^*= \Psi(a)^*\Psi(a).$ Thus
\[
\Psi(a)\Psi(a)^*\leq\Psi(aa^*)\leq  \Psi(a^*a)=\Psi(a)^*\Psi(a)=\Psi(a)\Psi(a)^*.
\]
These inequalities force $\Psi(aa^*)= \Psi(a^*a)$ and the proof is complete.
\end{proof}

\section{Residually finite-dimensional operator algebras}\label{S:RFD}

In the previous section, we investigated finite-dimensional operator algebras, and showed among other things that they are RFD (Corollary \ref{C:fdimRFD}). In this section, we study general RFD operator algebras. The first order of business is to obtain a more flexible characterization of residual finite-dimensionality. We emphasize one more time that, as seen in Example \ref{E:fdoa}, finite-dimensional operator algebras are not necessarily completely isometrically embeddable  in a finite-dimensional $\rC^*$-algebra, so the next fact is not obvious at first glance.

\begin{theorem}\label{T:fdimvsfdim}
Let $\A$ be an operator algebra. Consider the following statements.
\begin{enumerate}

\item[\rm{(i)}] There is a collection $\{\B_\lambda\}_{\lambda\in \Lambda}$ of finite-dimensional operator algebras and a completely isometric homomorphism $\rho:\A\to \prod_{\lambda\in \Lambda}\B_\lambda$.

\item[\rm{(ii)}] There is a collection $\{\fB_\lambda\}_{\lambda\in \Lambda}$ of finite-dimensional $\rC^*$-algebras and a completely isometric homomorphism $\rho:\A\to \prod_{\lambda\in \Lambda}\fB_\lambda$.

\item[\rm{(iii)}] For every $d\in\bN$ and every $A\in \bM_d(\A)$, there is a finite-dimensional operator algebra $\B$ and a completely contractive homomorphism  $\pi:\A\to\B$ such that $\|\pi^{(d)}(A)\|=\|A\|$.

\item[\rm{(iv)}] For every $d\in\bN$ and every $A\in \bM_d(\A)$, there is a finite-dimensional $\rC^*$-algebra $\fB$ and a completely contractive homomorphism  $\pi:\A\to\fB$ such that $\|\pi^{(d)}(A)\|=\|A\|$.

\end{enumerate}

Then, \emph{(i)} and \emph{(ii)} are equivalent, and \emph{(iii)} and \emph{(iv)} are equivalent.
\end{theorem}
\begin{proof}
It is trivial that (ii) implies (i) and that (iv) implies (iii).

Assume that (i) holds. For each $\lambda \in \Lambda$, we may apply Corollary \ref{C:fdimRFD} to find a set of positive integers $\{r_\mu:\mu\in \Omega_\lambda\}$ and a completely isometric homomorphism 
\[
\pi_\lambda:\B_\lambda\to \prod_{\mu\in \Omega_\lambda}\bM_{r_\mu}.
\]
The map 
\[
(\oplus_{\lambda\in \Lambda}\pi_\lambda)\circ \rho:\A\to \prod_{\lambda\in\Lambda}\prod_{\mu\in \Omega_\lambda}\bM_{r_\mu}
\]
is a completely isometric homomorphism, and thus  (ii) follows.

Finally, assume that (iii) holds. Let $d\in \bN$ and $A\in \bM_d(\A)$. Choose a finite-dimensional  operator algebra $\B$ and a completely contractive homomorphism $\pi:\A\to\B$ such that $\|\pi^{(d)}(A)\|=\|A\|$. Apply now Theorem \ref{T:fdimRFDnorm} to find a finite-dimensional Hilbert space $\fF$ and a completely contractive homomorphism $\sigma:\B\to B(\fF)$ such that $\|\sigma^{(d)}(\pi^{(d)}(A))\|=\|\pi^{(d)}(A)\|$. Thus, $\sigma\circ\pi:\A\to B(\fF)$ is a completely contractive homomorphism with $\|(\sigma\circ\pi)^{(d)}(A)\|=\|A\|$, and $B(\fF)$ is a finite-dimensional $\rC^*$-algebra. We conclude that (iv) holds.
\end{proof}

The next development is inspired by \cite{courtney2017}. We aim to identify elements in an RFD operator algebra, the norm of which can be attained in a finite-dimensional representation. For simplicity, we restrict our attention to the separable case. Thus, we fix a sequence $(r_n)_n$ of positive integers, and for each $n\in \bN$ we let 
\[
\gamma_n:\prod_{m=1}^\infty \bM_{r_m}\to \bM_{r_n}
\]
denote the natural projection. Put 
\[
\fL=\oplus_{n=1}^\infty \bM_{r_n}=\left\{t\in \prod_{n=1}^\infty \bM_{r_n}:\lim_{n\to\infty}\|\gamma_n(t)\|=0\right\}.
\]
Before stating the result, we record a standard calculation.

\begin{lemma}\label{L:limsup}
Let $\fA\subset \prod_{n=1}^\infty \bM_{r_n}$ be a $\rC^*$-algebra and let 
\[
\kappa:\fA\to \fA/(\fA\cap \fL)
\]
 be the quotient map. Then, we have that
\[
\|\kappa^{(d)}(A)\|=\limsup_{n\to\infty}\|\gamma_n^{(d)}(A)\|
\]
for every $A\in \bM_d(\fA)$ and every $d\in \bN$.
 \end{lemma}
\begin{proof}
For convenience, we put $\fK=\fA\cap \fL$. Fix $d\in \bN$ and $A\in \bM_d(\fA)$. Let $T=[t_{jk}]_{j,k=1}^d\in \bM_d(\fK)$ and let $\delta>0$. Then, there is an $N\in \bN$ such that 
\[
\|\gamma_n(t_{jk})\|<\delta/d^2, \quad 1\leq j,k\leq d
\]
if $n\geq N$. We have
\begin{align*}
\|A+T\|&=\sup_{n\in \bN}\|\gamma^{(d)}_n(A+T)\|\\
&\geq \sup_{n\geq N}\|\gamma^{(d)}_n(A+T)\|\\
&\geq \sup_{n\geq N}\|\gamma^{(d)}_n(A)\|-\delta.
\end{align*}
Thus, we find
\[
\|A+T\|\geq \inf_{m\in \bN}\sup_{n\geq m}\|\gamma^{(d)}_n(A)\|-\delta
\]
and since $T\in  \bM_d(\fK)$ is arbitrary, this means that
\[
\|\kappa^{(d)}(A)\|\geq \inf_{m\in \bN}\sup_{n\geq m}\|\gamma^{(d)}_n(A)\|-\delta.
\]
Next, we observe that the map
\[
(\fA+\fL)/\fL\to \fA/\fK
\]
defined as
\[
(a+t)+\fL\mapsto a+\fK, \quad a\in \fA, t\in \fL
\]
is a $*$-isomorphism \cite[Corollary II.5.1.3]{blackadar2006}, so that
\[
\|\kappa^{(d)}(A)\|=\inf\{\|A+L\|:L\in \bM_d(\fL)\}.
\]
There is $M\in \bN$ with the property that 
\[
\sup_{n\geq M}\|\gamma_n^{(d)}(A)\|\leq \inf_{m\in \bN}\sup_{n\geq m}\|\gamma^{(d)}_n(A)\|+\delta.
\]
 For each $1\leq j,k\leq d$, let $t_{jk}\in \fL$ be defined as
\[
\gamma_n(t_{jk})=\begin{cases}
-\gamma_n(a_{jk}) & \text{ if } n<M,\\
0 & \text{ otherwise}.
\end{cases}
\]
Put $T=[t_{jk}]_{j,k}\in \bM_d(\fL)$ and note that
\begin{align*}
\|A+T\|&=\sup_{n\in \bN}\|\gamma_n^{(d)}(A+T)\|\\
&=\sup_{n\geq M}\|\gamma_n^{(d)}(A)\|\\
& \leq  \inf_{m\in \bN}\sup_{n\geq m}\|\gamma^{(d)}_n(A)\|+\delta.
\end{align*}
Thus,
\[
\|\kappa^{(d)}(A)\|=\inf\{\|A+L\|:L\in \bM_d(\fL)\}\leq  \inf_{m\in \bN}\sup_{n\geq m}\|\gamma^{(d)}_n(A)\|+\delta.
\]
We conclude that
\[
\inf_{m\in \bN}\sup_{n\geq m}\|\gamma^{(d)}_n(A)\|-\delta\leq \|\kappa^{(d)}(A)\|\leq \inf_{m\in \bN}\sup_{n\geq m}\|\gamma^{(d)}_n(A)\|+\delta.
\]
Since $\delta>0$ is arbitrary, the proof is complete.
\end{proof}

We can now identify a sufficient condition for the norm of an element to be attained in a finite-dimensional representation. Roughly speaking, the condition says that the element must be small at infinity.

\begin{theorem}\label{T:RFDnorm}
Let $\A\subset \prod_{n=1}^\infty \bM_{r_n}$ be an operator algebra and let 
\[
\kappa:\A\to \A/(\fL\cap \A)
\]
denote the quotient map. Let $d\in \bN$ and let $A\in \bM_d(\A)$ be an element with the property that $\|\kappa^{(d)}(A)\|<\|A\|$. Then, there is a finite-dimensional Hilbert space $\fF$ and a completely contractive homomorphism $\pi:\A\to B(\fF)$ such that $\|\pi^{(d)}(A)\|=\|A\|$.
\end{theorem}
\begin{proof}
By virtue of Lemma \ref{L:limsup}, we find that
\[
\inf_{m\in \bN}\sup_{n\geq m}\|\gamma^{(d)}_n(A)\|<\|A\|.
\]
On the other hand, we know that
\[
\|A\|=\sup_{n\in \bN}\|\gamma^{(d)}_n(A)\|.
\]
Thus, there is $m\in \bN$ with the property that
\[
\max_{1\leq n\leq m}\|\gamma^{(d)}_n(A)\|=\|A\|
\]
which clearly implies the desired statement.
\end{proof}

If the sizes of the matrix algebras are bounded, more can be said.

\begin{corollary}\label{C:RFDnorm}
Assume that the sequence $(r_n)_n$ is bounded and let $\A\subset \prod_{n=1}^\infty \bM_{r_n}$ be an operator algebra.  For every $d\in \bN$ and every $A\in \bM_d(\A)$, there is a finite-dimensional Hilbert space $\fF$ and a completely contractive homomorphism $\pi:\A\to B(\fF)$ such that $\|\pi^{(d)}(A)\|=\|A\|$.
\end{corollary}
\begin{proof}
Fix $d\in \bN$ and $A\in\bM_d(\A)$. If
\[
\inf_{m\in \bN}\sup_{n\geq m}\|\gamma^{(d)}_n(A)\|<\|A\|
\]
then the conclusion follows from Theorem \ref{T:RFDnorm}. If, on the other hand,
\[
\inf_{m\in \bN}\sup_{n\geq m}\|\gamma^{(d)}_n(A)\|=\|A\|
\]
then the conclusion follows as in the proof of \cite[Proposition 3.5]{CM2017rfd}.
\end{proof}

We note that there are examples of RFD $\rC^*$-algebras containing elements that do not attain their norms in a finite-dimensional representation \cite[Theorem 4.4]{courtney2017}. The full $\rC^*$-algebra of the free group on two generators is such an example.

\subsection{Examples of RFD operator algebras}\label{SS:Examples}

The remainder of this section is devoted to studying residual finite-dimensionality in various concrete examples. First, we  mention that the property of an operator algebra being RFD isn't preserved by crossed products with groups. This is an immediate consequence of the so-called \emph{Takai duality} \cite[Theorem 4.4]{KRmem}. We refer the interested reader to \cite{KRmem} for details on these topics. For now, we turn to a widely studied class of operator algebras: the tensor algebras of $\rC^*$-correspondences \cite{pimsner1997},\cite{MS1998}. We are interested in determining when these are RFD.

We briefly recall the relevant definitions; more details can be found in \cite{MS1998} or \cite[Section 4.6]{BO2008} for instance. Let $\fA$ be a C$^*$-algebra and let $X$ be a right Hilbert $\fA$-module. Denote by $\L(X)$ the $\rC^*$-algebra of  adjointable operators on $X$.  If in addition there is a non-degenerate $*$-homomorphism $\varphi_X : \fA \rightarrow \L(X)$ (which we think of as a left action of $\fA$ on $X$), then $X$ is called a {\em $\rC^*$-correspondence over $\fA$}. When there is no danger of confusion, $\phi_X$ is not mentioned explicitly. 

It is possible to form direct sums and tensor products of $\rC^*$-correspondences. Given a $\rC^*$-correspondence $X$, we can then define the \emph{Fock correspondence over $X$} as 
\[
\F_X := \fA \oplus \bigoplus_{n=1}^\infty X^{\otimes n}.
\]
The {\em tensor algebra} of $X$ is the norm closed operator algebra $\T_{X}^+\subset \L(\F_X)$ generated by the image of the creation map $t_\infty :X \rightarrow \L(\F_X)$ and the non-degenerate $*$-homomorphism $\rho_\infty : \fA \rightarrow \L(\F_X)$ that gives rise to the natural left module action. The \emph{Toeplitz algebra} is defined as $\T_X=\rC^*(\T^+_X)$. In our analysis of the residual finite-dimensionality of $\T^+_X$, we will not require the precise definitions of the maps $t_\infty$ and $\rho_\infty$, but we will require the following important properties. First we have that
\[
\rho_\infty(a)t_\infty(x)=t_\infty(\phi_X(a) x) , \quad   t_\infty(x)\rho_\infty(a)=t_\infty( xa)
\]
for every $a\in \fA,x\in X$, and
\[
 t_\infty(x_1)^* t_\infty(x_2)=\rho_\infty(\langle  x_1,x_2\rangle)
\]
for every $x_1,x_2\in X$. In fact, this says that the pair $(\rho_\infty,t_\infty)$ is an  \emph{isometric representation} of the $\rC^*$-correspondence $X$.  One consequence of this is that $\T^+_X$ is the closure of the subspace $\T^0_X\subset \L(\F_X)$ spanned by $\rho_\infty(\fA)$ and elements of the form 
\[
t_\infty(x_1)t_\infty(x_2)\cdots t_\infty(x_n)
\]
for some $n\in \bN$ and $x_1,\ldots,x_n\in X$. 

The following elementary observation, inspired by the argument given in \cite[Theorem 4.1]{pestov1994}, will be very useful for us throughout the paper.

\begin{lemma}\label{L:fdiminvsub}
Let $\A\subset B(\fH)$ be a finite-dimensional operator algebra and let $d\in \bN$. For each $1\leq \nu\leq d$, let $r_\nu\in \bN$ and let $x^{(\nu)}_{j,1},\ldots, x^{(\nu)}_{j,N_j}\in B(\fH)$ be arbitrary elements for every $1\leq j\leq r_\nu$. For each $1\leq \nu\leq d$, put
\[
s_\nu=\sum_{j=1}^{r_\nu} x^{(\nu)}_{j,1}\cdots x^{(\nu)}_{j,N_j}.
\]
Let $\Xi\subset \H$ be a finite subset. Then, there is a finite-dimensional subspace $\fF\subset \fH$ containing $\Xi$ that is invariant for $\A$ and such that 
\[
s_\nu\xi=\sum_{j=1}^{r_\nu} P_{\fF}x^{(\nu)}_{j,1}P_{\fF} x^{(\nu)}_{j,2}P_{\fF}\cdots P_{\fF}x^{(\nu)}_{j,N_j}\xi
\]
for every $1\leq \nu\leq d$ and every $\xi\in \Xi$.

\end{lemma}
\begin{proof}
Define $\fF_0 \subset \fH$ to be the finite-dimensional subspace spanned by $\Xi$ and by $x^{(\nu)}_{j,k}x^{(\nu)}_{j,k+1}\cdots x^{(\nu)}_{j,N_j}\Xi$ for $1\leq \nu\leq d,1\leq j\leq r_\nu, 1\leq k\leq N_j.$ Put $\fF=\ol{\fF_0+\spn\A\fF_0}$. By construction, we see that $\fF$ is invariant for $\A$, and it is finite-dimensional since $\A$ and $\fF_0$ are. For each $1\leq \nu\leq d$ and $\xi\in \Xi$, we observe now that 
\begin{align*}
\sum_{j=1}^{r_\nu}  P_{\fF}x^{(\nu)}_{j,1}P_{\fF} x^{(\nu)}_{j,2}P_{\fF}\cdots P_{\fF}x^{(\nu)}_{j,N_j}\xi&=\sum_{j=1}^{r_\nu}P_{\fF}x^{(\nu)}_{j,1}P_{\fF} x^{(\nu)}_{j,2}P_{\fF}\cdots P_{\fF}x^{(\nu)}_{j,N_j-1}x^{(\nu)}_{j,N_j}\xi\\
&=\sum_{j=1}^{r_\nu}P_{\fF}x^{(\nu)}_{j,1}P_{\fF} x^{(\nu)}_{j,2}P_{\fF}\cdots P_{\fF}x^{(\nu)}_{j,N_j-2}x^{(\nu)}_{j,N_j-1}x^{(\nu)}_{j,N_j}\xi
\end{align*}
and proceeding inductively we find
\[
\sum_{j=1}^{r_\nu}  P_{\fF}x^{(\nu)}_{j,1}P_{\fF} x^{(\nu)}_{j,2}P_{\fF}\cdots P_{\fF}x^{(\nu)}_{j,N_j}\xi=s_\nu\xi.
\]
\end{proof}

We can now show that the tensor algebra of a $\rC^*$-correspondence is RFD whenever the underlying $\rC^*$ algebra is finite-dimensional. In fact, we can say something a bit finer.

\begin{theorem}\label{T:fdcorrespondence}
Let $\fA$ be a finite-dimensional $\rC^*$-algebra and let $X$ be a $\rC^*$-corres\-pondence over $\fA$. Then, for every $d\in \bN$ and every $S\in \bM_d(\T^0_X)$ there is a finite-dimensional Hilbert space $\fF$ and a completely contractive homomorphism $\pi:\T^+_X\to B(\fF)$ with the property that $\|\pi^{(d)}(S)\|=\|S\|$. In particular, the tensor algebra $\T_{X}^+$ is RFD.
\end{theorem}
\begin{proof}

Fix $d\in \bN$ and $S=[s_{\mu,\nu}]_{\mu,\nu}\in \bM_d(\T^0_X)$. For each $1\leq \mu,\nu\leq d$, there is a positive integer $r_{\mu,\nu}$, an element $a_{\mu,\nu}\in \fA$ and elements $x^{(\mu,\nu)}_{j,1},\ldots, x^{(\mu,\nu)}_{j,N_j}\in X$  for every $1\leq j\leq r_{\mu,\nu}$ such that
\[
s_{\mu,\nu}=\rho_\infty(a_{\mu,\nu})+\sum_{j=1}^{r_{\mu,\nu}} t_\infty(x^{(\mu,\nu)}_{j,1})\cdots t_\infty( x^{(\mu,\nu)}_{j,N_j}).
\]
Note now that $\fA$ is unital, and since $\rho_\infty$ is non-degenerate, it must also be unital. In particular, the Toeplitz algebra $\T_X$ is unital. We may choose  a state $\psi$ on $\bM_d(\T_X)$  such that $\psi(S^*S)=\|S\|^2$. Let $\sigma:\bM_d(\T_X)\to B(\fH)$ be the associated GNS representation with cyclic unit vector $\xi\in \fH$. Then, we see that
\[
\|\sigma(S)\xi\|^2=\psi(S^*S)=\|S\|^2.
\]
There is a Hilbert space $\fH'$, a unitary operator $U:\fH\to \fH'^{(d)}$ and a unital $*$-homomorphism $\tau: \T_X\to B(\fH')$ such that
\[
U\sigma(R)U^*=\tau^{(d)}(R), \quad R\in \bM_d(\T_X).
\]
Write 
\[
U\xi=\xi_1\oplus \xi_2\oplus \ldots \oplus \xi_d\in \fH'^{(d)}
\]
and set $\Xi=\{\xi_1,\ldots,\xi_d\}\subset \fH'$.  By Lemma \ref{L:fdiminvsub}, there a finite-dimensional subspace $\fF\subset \fH'$ containing $\Xi$ that is invariant for $\tau(\rho_\infty(\fA))$ and such that 
\begin{align*}
&\tau(s_{\mu,\nu})\xi_m\\
&=P_{\fF}\tau(\rho_\infty(a_{\mu,\nu}))\xi_m+\sum_{j=1}^{r_{\mu,\nu}} P_{\fF}\tau(t_\infty(x^{(\mu,\nu)}_{j,1}))P_{\fF}\tau(t_\infty(x^{(\mu,\nu)}_{j,2}))P_{\fF}\cdots P_{\fF} \tau(t_\infty(x^{(\mu,\nu)}_{j,N_j}))\xi_m
\end{align*}
for every $1\leq \mu,\nu\leq d$ and every $1\leq m\leq d$. Define now maps $t:X\to B(\fF)$ and $\rho:\fA\to B(\fF)$ as
\[
t(x)=P_{\fF}\tau(t_\infty(x))|_{\fF}, \quad \rho(a)=P_{\fF}\tau(\rho_\infty(a))|_{\fF}
\]
for every $a\in \fA$ and every $x\in X$. We see that
\[
\tau(s_{\mu,\nu})\xi_m=\rho(a_{\mu,\nu})\xi_m+\sum_{j=1}^{r_{\mu,\nu}} t(x^{(\mu,\nu)}_{j,1}) t(x^{(\mu,\nu)}_{j,2})\cdots  t(x^{(\mu,\nu)}_{j,N_j})\xi_m
\]
for every $1\leq \mu,\nu\leq d$ and every $1\leq m\leq d$. 
Since $\fF$ is invariant for $\tau(\rho_\infty(\fA))$, we obtain that $\rho$ is a unital $*$-homomorphism, and that
\begin{align*}
t(x)\rho(a)&=P_{\fF}\tau(t_\infty(x))P_{\fF}\tau(\rho_\infty(a))|_{\fF}=P_{\fF}\tau(t_\infty(x)\rho_\infty(a))|_{\fF}\\
&=P_{\fF}\tau(t_\infty (xa))|_{\fH}=t(xa)
\end{align*}
while
\begin{align*}
\rho(a)t(x)&=P_{\fF}\tau(\rho_\infty(a))P_{\fF}\tau(t_\infty(x))|_{\fF}=P_{\fF}\tau(\rho_\infty(a)t_\infty(x))|_{\fF}\\
&=P_{\fF}\tau(t_\infty (\phi_X(a)x))|_{\fH}=t(\phi_X(a)x)
\end{align*}
for every $a\in \fA$ and every $x\in X$. By \cite[Theorem 3.10]{MS1998}, we conclude that there is a completely contractive homomorphism $\pi:\T^+_X\to B(\fF)$ such that
\[
\pi(\rho_\infty(a))=\rho(a), \quad \pi(t_\infty(x))=t(x)
\]
for every $a\in \fA$ and every $x\in X$. In particular, we see that
\[
\pi(s_{\mu,\nu})\xi_m=\rho(a_{\mu,\nu})\xi_m+\sum_{j=1}^{r_{\mu,\nu}} t(x^{(\mu,\nu)}_{j,1}) t(x^{(\mu,\nu)}_{j,2})\cdots  t(x^{(\mu,\nu)}_{j,N_j})\xi_m=\tau(s_{\mu,\nu})\xi_m
\]
for every $1\leq \mu,\nu\leq d$ and every $1\leq m\leq d$. 
This means that
\[
\pi^{(d)}(S)U\xi=\tau^{(d)}(S)U\xi
\]
hence
\[
\|\pi^{(d)}(S)\|\geq \|\tau^{(d)}(S) U\xi\|=\|U\sigma(S)\xi\|=\|S\|
\]
and $\|\pi^{(d)}(S)\|=\|S\|$, which establishes the first statement. The second follows immediately, since $\T^0_X$ is dense in $\T^+_X$.
\end{proof}

Given a $\rC^*$-correspondence $X$ over $\fA$, it is known that $\T_{X}^+ \cap (\T_{X}^+)^* = \fA$ \cite[page 10]{KRmem}. In view of Theorem \ref{T:fdcorrespondence}, one may wonder whether, for a general operator algebra $\A$, the fact that the $\rC^*$-algebra $\A\cap \A^*$ is finite-dimensional  implies that $\A$ is necessarily RFD. We show that this is not the case. 

\begin{example}\label{E:diagonal}
Let $\fH$ be an infinite-dimensional separable Hilbert space. Let $\A\subset B(\fH^{(3)})$ be the operator algebra consisting of elements of the form
\[
\begin{bmatrix}0 &r&s\\0&0&t\\0&0&0 \end{bmatrix}, \quad r,s,t\in B(\fH).
\]
Then, $\A\cap \A^*=\{0\}$ yet we claim that $\A$ is not RFD. 

To see this, let $\theta:\A\to B(\fF)$ be a completely contractive homomorphism such that
\[
\theta\left(\begin{bmatrix}0 &0&I\\0&0&0\\0&0&0 \end{bmatrix} \right)\neq 0.
\]
There are completely contractive linear maps 
\[
\theta_{12},\theta_{13},\theta_{23} : B(\fH) \rightarrow B(\fF)
\]
such that
\[
\theta_{12}(t)=\theta\left(\begin{bmatrix}0 &t&0\\0&0&0\\0&0&0 \end{bmatrix} \right), 
\theta_{13}(t)=\theta\left(\begin{bmatrix}0 &0&t\\0&0&0\\0&0&0 \end{bmatrix} \right),
\theta_{23}(t)=\theta\left(\begin{bmatrix}0 &0&0\\0&0&t\\0&0&0 \end{bmatrix} \right)
\]
for every $t\in B(\fH)$. Since $\theta$ is multiplicative, we find
\[
\theta_{ij}(r)\theta_{ij}(s) = 0, \quad r,s\in B(\fH)
\]
for every $i,j$, while
\[
\theta_{12}(r)\theta_{23}(s) = \theta_{13}(rs), \quad r,s\in B(\fH).
\]
By choice of $\theta$, we know that $\theta_{13}(I)\neq 0$. Choose a sequence of partial isometries $(v_n)_n$ in $B(\fH)$ with the property that $v_n v_n^*=I$ and $v_n v_k^*=0$ for every  $k\neq n$ and for every $n\in \bN$. 
Let $c_1,\ldots,c_m\in \bC$ such that 
\[
\sum_{n=1}^m c_n \theta_{12}(v_n)=0.
\]
Thus, for each $1\leq k \leq m$ we find
\begin{align*}
c_k \theta_{13}(I)&=c_k \theta_{13}( v_k v_k^*)=\sum_{n=1}^m c_n \theta_{13}(v_nv_k^*)\\
& =\left(\sum_{n=1}^m c_n \theta_{12}(v_n)\right)\theta_{23}(v_k^*) =0
\end{align*}
which forces $c_k=0$. Hence, $\{\theta_{12}(v_n):n\in \bN\}$  is a  linearly independent set in $B(\fF)$, so that $\fF$ is infinite-dimensional.
\qed
\end{example}

As an application, we next consider a noteworthy special class of $\rC^*$-corres\-pondences.  Let $G = (E,V,s,r)$ be a countable directed graph. Let $c_{00}(E)$ denote the algebra of finitely supported functions on $E$. Let $c_0(V)$ denote the $\rC^*$-algebra obtained by taking the uniform closure of $c_{00}(V)$. We may define a $c_0(V)$-bimodule structure on $c_{00}(E)$ by setting
\[
(a\cdot x \cdot b)(e)=a(r(e))x(e)b(s(e)), \quad e\in E
\]
for every $a,b\in c_0(V)$ and every $x\in c_{00}(E)$. We may also define a $c_0(V)$-valued inner product on $c_{00}(E)$ by setting
\[
\langle x,y \rangle(v)=\sum_{e\in s^{-1}(v)}\ol{x(e)}y(e), \quad v\in V
\]
for every $x,y\in c_{00}(E)$. Upon applying a standard completion procedure, we obtain a $\rC^*$-correspondence $X_G$ over $c_0(V)$, called the \emph{graph correspondence of $G$}.  See \cite[Example 4.6.13]{BO2008} for more detail.

For our purposes, we will require the following observations. Fix a finite subset $F\subset V$. Extending functions by $0$ outside of $F$ and $V\setminus F$ respectively, we view $c_{0}(F)$ and $c_0(V\setminus F)$ as closed two-sided ideals of $c_{0}(V)$ such that 
\[
c_0(V)=c_0(F)\oplus c_0(V\setminus F).
\] 
A similar decomposition holds for the $\rC^*$-correspondence $X_G$. Indeed,  let
\[
E_F = \{e\in E : s(e), r(e) \in F\}.
\]
Extending functions by $0$ outside of $E_F$ and $E\setminus E_F$ respectively, we see that $c_{00}(E_F)$ and $c_{00}(E\setminus E_F)$ are submodules of $c_{00}(E)$. Furthermore, for $x,y\in c_{00}(E)$ we note that
\[
\langle x,y\rangle=\langle x|_{E_F},y|_{E_F}\rangle+\langle x|_{E\setminus E_F},y|_{E\setminus E_F}\rangle.
\]
Let $H=(E_F,F,s,r)$ and let $Y\subset X_G$ denote the closure of $c_{00}(E\setminus E_F)$. Then, $X_{H}$ and $Y$ are submodules of $X_G$ with
\[
X_G=X_{H}\oplus Y.
\]
With these preliminaries out of the way, we can prove the following.

\begin{theorem}\label{T:graph}
Let $G = (E,V,s,r)$ be a directed graph  and let $X_G$ be the associated graph correspondence. Then, $\T^+_{X_G}$ is RFD.
\end{theorem}
\begin{proof}
We use the same notation as in the discussion preceding the theorem. Fix a finite subset $F\subset V$ and let 
\[
\rho_\infty^F:c_0(F)\to \L(\F_{X_{H}}) \qand t_\infty^F:X_{H}\to \L(\F_{X_{H}})
\]
denote the usual maps giving rise to the tensor algebra of $X_H$. Let $\gamma_F :c_0(V)\to c_0(F)$ and $\delta_F:X_G\to X_H$ be the projections onto the appropriate submodule obtained via restriction, as above.
It is easily seen that
\[
\gamma_F(a)\delta_F(x)=\delta_F(a\cdot x), \quad \delta_F(x)\gamma_F(a)=\delta_F(x\cdot a)
\]
for every $a\in c_0(V)$ and $x\in X_G$. 
By \cite[Theorem 3.10]{MS1998}, we find a completely contractive homomorphism $\pi_F:\T^+_{X_G}\to \T^+_{X_F}$ such that
\[
\pi_F(\rho_\infty(a)) =\rho^F_\infty(\gamma_F(a)), \quad \pi_F(t_\infty(x))=t^F_\infty( \delta_F(x) )
\]
for every $a\in c_0(V)$ and $x\in X_G$. 

Next, let $\D=\ol{\rho_\infty(c_0(F))\F_{X_G}}$. Since $c_0(F)$ is an ideal in $c_0(V)$, we infer that  $\D$ is invariant for $\rho_\infty(c_0(V))$, and that the map 
\[
b\mapsto b|_\D, \quad b\in c_0(F)
\]
is a non-degenerate $*$-homomorphism. We claim now that $\D$ is invariant for $t_\infty(X_F)$ and $t_\infty(X_F)^*$. To see this, first note that if $x\in X_F$ and $a\in c_0(V)$, then  
\[
a\cdot x=\gamma_F(a) \cdot x, \quad x\cdot  a=x\cdot \gamma_F(a)
\]
by definition of $E_F$. Let  $(e_\lambda)_{\lambda \in \Lambda}$ be a contractive approximate identity for $c_0(V)$ and recall that $\rho_\infty$ is non-degenerate so $\rho_\infty(c_0(V))\F_{X_G}$ is dense in $\F_{X_G}$. In particular, this means that 
\[
\lim_{\lambda\in \Lambda}\rho_\infty(e_\lambda) h=h
\]
 for every $h\in \F_{X_G}$. Thus, for $y\in X_F$ and $h\in\F_{X_G}$ we find
\begin{align*}
t_\infty(y)h&=\lim_{\lambda\in \Lambda} \rho_\infty(e_\lambda)t_\infty(y)h=\lim_{\lambda\in \Lambda} t_\infty(e_\lambda \cdot y)h\\
&=\lim_{\lambda\in \Lambda} t_\infty(\gamma_F(e_\lambda) \cdot y)h=\lim_{\lambda\in \Lambda}\rho_\infty(\gamma_F(e_\lambda)) t_\infty(y)h
\end{align*}
and
\begin{align*}
t_\infty(y)^*h&=\lim_{\lambda\in \Lambda} \rho_\infty(e_\lambda)t_\infty(y)^*h=\lim_{\lambda\in \Lambda}( t_\infty(y)\rho_\infty(e_\lambda^*))^*h\\
&=\lim_{\lambda\in \Lambda}t_\infty(y\cdot e_\lambda^*)^* h=\lim_{\lambda\in \Lambda}t_\infty(y\cdot \gamma_F(e_\lambda^*))^*h\\
&=\lim_{\lambda\in \Lambda}(t_\infty(y)\rho_\infty(\gamma_F(e_\lambda^*)))^*h\\
&=\lim_{\lambda\in \Lambda}\rho_\infty(\gamma_F(e_\lambda)) t_\infty(y)^*h
\end{align*}
whence $t_\infty(y)h,t_\infty(y)^*h\in \D$. We conclude that 
\[
t_\infty(X_F)\F_{X_G}\subset \D \qand t_\infty(X_F)^*\F_{X_G}\subset \D
\]
so in particular $\D$ is invariant for  $t_\infty(X_F)$ and $t_\infty(X_F)^*$, and the claim is established.  Invoking \cite[Theorem 3.10]{MS1998} again, we find a completely contractive homomorphism $\sigma_F: \T^+_{X_F}\to \L(\D)$ such that
\[
\sigma_F(\rho^F_\infty(b)) =\rho_\infty(b)|_\D, \quad \sigma_F(t^F_\infty(y))=t_\infty(y)|_\D
\]
for every $b\in c_0(F)$ and $y\in X_F$. Denote by $\T_F\subset \T^+_{X_G}$ the algebra generated by $\rho_\infty(c_0(F))$ and $t_\infty(X_F)$. Then, we find that 
\[
\sigma_F\circ \pi_F(s)=s|_\D
\]
for every $s\in \T_F$.
Moreover, for $s\in \T_F$ we find
\[
s\rho_\infty(e_\lambda)=s\rho_\infty(\gamma_F(e_\lambda))
\]
for every $\lambda\in \Lambda$. Since $\rho_\infty(c_0(V))\F_{X_G}$ is dense in $\F_{X_G}$, we conclude that the map
\[
s\mapsto s|_\D, \quad s\in \T_F
\]
is completely isometric, and thus $\sigma_F\circ \pi_F$ is completely isometric on $\T_F$. In turn, this forces $\pi_F$ to be completely isometric on $\T_F$.

Finally, define 
\[
\pi:\T^+_{X_G}\to \prod_{F\subset V \textrm{finite}}\T^+_{X_F}
\]
as
\[
\pi(s)=\bigoplus_{F\subset V \textrm{finite}} \pi_F(s), \quad s\in \T^+_{X_G}.
\]
Clearly, $\pi$ is a completely contractive homomorphism. For each finite subset $F\subset V$, the $\rC^*$-algebra $c_0(F)$ is finite-dimensional, whence $\T^+_{X_F}$ is RFD by Theorem \ref{T:fdcorrespondence}. Thus, to complete the proof it suffices to show that $\pi$ is completely isometric. To see this, fix $S\in \bM_d(\T^+_{X_G})$ and $\eps>0$. Note now that
\[
c_0(V)=\ol{\bigcup_{F\subset V \textrm{finite}}c_0(F)}, \quad X_G=\ol{\bigcup_{F\subset V \textrm{finite}}X_F}.
\]
Thus, we may find a finite subset $F\subset V$ and an element $S'\in \bM_d(\T_F)$ such that $\|S-S'\|<\eps$.
Using that $\pi$ is completely isometric on $\T_F$, we obtain
\[
\|\pi_F^{(d)}(S)\|\geq \|\pi^{(d)}(S')\|-\eps=\|S'\|-\eps\geq \|S\|-2\eps.
\]
This shows that $\pi$ is completely isometric.
\end{proof}

We have demonstrated that there is a large class of tensor algebras of C$^*$-correspondences that are RFD. Ultimately, one would wish to determine whether it is enough for the $\rC^*$-algebra $\fA$ to be RFD in order for the tensor algebra of a $\rC^*$-correspondence over $\fA$ to be RFD. We end this section with one more positive step towards answering this question.

\begin{theorem}\label{T:X=C}
Let $\fA$ be a $\rC^*$-algebra which we view as a $\rC^*$-correspondence over itself. Then, $\rC^*_e(\T_{\fA}^+)$ is RFD if and only if $\fA$ is RFD.
\end{theorem}

\begin{proof}
Assume first that $\rC^*_e(\T_{\fA}^+)$ is RFD. Therefore, $\T_{\fA}^+$ is also RFD. It is easily verified that $\rho_\infty$ is injective in this situation, so that $\fA\cong \rho_\infty(\fA)$. Since $\rho_\infty(\fA)$ is a subalgebra of $\T_{\fA}^+$, we infer that $\fA$ is RFD. Conversely, assume that $\fA$ is RFD. In \cite[Example 2.6]{MS1998} it is pointed out that $\T_{\fA}^+$ is completely isometrically isomorphic to $\fA \times_{\id} \mathbb Z^+$. In turn, by \cite[Corollary 6.9]{MS1998} (see also \cite[Theorem 4.1]{DFK2017survey}), we obtain
\[
\rC^*_e(\T_{\fA}^+)\cong \fA \otimes_{\max} \rC(\bT).
\]
Using that $\rC(\bT)$ is a nuclear $\rC^*$-algebra and invoking \cite[Proposition 12.5]{paulsen2002}, we infer
\[
\rC^*_e(\T_{\fA}^+)\cong \rC(\bT; \fA).
\]
Now, $\fA$ is assumed to be RFD, and thus so is  $\rC(\bT; \fA)$. We conclude that $\rC^*_e(\T_{\fA}^+)$ is RFD.

 \end{proof}


\section{Residual finite-dimensionality of the maximal $\rC^*$-cover}\label{S:C*max}

In this section, we explore the residual finite-dimensionality of the maximal $\rC^*$-cover of an operator algebra $\A$. Since the embedding $\mu:\A\to \rC^*_{\max}(\A)$ is a completely isometric homomorphism,  $\A$ being RFD is a necessary condition for $\rC^*_{\max}(\A)$ to be RFD. In light of \cite[Theorem 4.1]{pestov1994}, this shows that residual finite-dimensionality of the maximal $\rC^*$-cover is more nuanced than that of the free $\rC^*$-algebra of an operator space. We begin our analysis by considering the case where $\A$ is finite-dimensional. For convenience, we let $\bC\langle \A,\A^*\rangle\subset \rC^*_{\max}(\A)$ denote the linear span of words in the elements of $\mu(\A)\cup \mu(\A)^*$.  Recall that $\rC^*_{\max}(\A)=\rC^*(\mu(\A))$, so that $\bC\langle \A,\A^*\rangle$ is dense in $\rC^*_{\max}(\A)$. The following argument is very similar to that used in the proof of Theorem \ref{T:fdcorrespondence}.

\begin{theorem}\label{T:C*maxfdimA}
Let $\A$ be a finite-dimensional operator algebra. Then, for every $s\in \bC\langle \A,\A^*\rangle$ there is a finite-dimensional Hilbert space $\fF$ and a $*$-homomorphism $\pi:\rC^*_{\max}(\A)\to B(\fF)$ with the property that $\|\pi(s)\|=\|s\|$. In particular, the algebra $\rC^*_{\max}(\A)$ is RFD.

\end{theorem}
\begin{proof}
Upon identifying $\A$ with $\mu(\A)$, we may assume that $\A\subset \rC^*_{\max}(\A)$. Assume further that $\rC^*_{\max}(\A)$ is concretely represented on some Hilbert space $\fH$. Fix $s\in\bC\langle \A,\A^*\rangle$. We may write
\[
s=\sum_{j=1}^r c^{(j)}_1 c^{(j)}_2\cdots c^{(j)}_{N_j}
\]
where $c^{(j)}_k\in \A\cup \A^*$ for $1\leq k\leq N_j, 1\leq j\leq r$. 
Let $\psi$ be a state on $\rC^*_{\max}(\A)+\bC I_{\fH}$ such that $\psi(s^*s)=\|s\|^2$. Let $\sigma:\rC^*_{\max}(\A)+\bC I_{\fH}\to B(\fH_\psi)$ be the associated GNS representation with cyclic unit vector $\xi\in \fH_\psi$. Then, we see that
\[
\|\sigma(s)\xi\|^2=\psi(s^*s)=\|s\|^2.
\]
By Lemma \ref{L:fdiminvsub}, there is a finite-dimensional subspace $\fF\subset \fH_\psi$ containing $\xi$ that is invariant for $\sigma(\A)$ and such that 
\[
\sigma(s)\xi=\sum_{j=1}^r P_{\fF}\sigma(c^{(j)}_1)P_{\fF} \sigma(c^{(j)}_2)P_{\fF}\cdots P_{\fF}\sigma(c^{(j)}_{N_j})\xi.
\]
Since $\fF$ is invariant for $\sigma(\A)$, the map
\[
a\mapsto P_{\fF}\sigma(a)|_{\fF}, \quad a\in \A
\]
is a completely contractive homomorphism, whence there is a $*$-homomorphism $\pi:\rC^*_{\max}(\A)\to B(\fF)$ such that
\[
\pi(a)=P_{\fF}\sigma(a)|_{\fF}, \quad a\in \A.
\]
In particular, we see that
\[
\pi(a^*)=P_{\fF}\sigma(a)^*|_{\fF}, \quad a\in \A
\]
and since each element $c^{(j)}_k$ belongs to $\A\cup \A^*$, we obtain
\begin{align*}
\pi(s)\xi&=\sum_{j=1}^r \pi(c^{(j)}_1)\pi(c^{(j)}_2)\cdots \pi(c^{(j)}_{N_j})\xi\\
&=\sum_{j=1}^r P_{\fF}\sigma(c^{(j)}_1)P_{\fF} \sigma( c^{(j)}_2)P_{\fF}\cdots P_{\fF}\sigma(c^{(j)}_{N_j})\xi\\
&=\sigma(s)\xi.
\end{align*}
This means that
\[
\|\pi(s)\|\geq \|\sigma(s) \xi\|=\|s\|
\]
whence $\|\pi(s)\|=\|s\|$, which establishes the first statement. The second follows immediately, since $\bC\langle \A,\A^*\rangle$ is dense in $\rC^*_{\max}(\A)$.

\end{proof}

The following fact, which is most likely known to experts, provides motivation for what is to come.

\begin{theorem}\label{T:C*maxdirectsum}
For each $n\in \bN$, let $\A_n$ be a unital operator algebra and let $\A=\oplus_{n=1}^\infty \A_n$. Then, $\rC^*_{\max}(\A)\cong \oplus_{n=1}^\infty \rC^*_{\max}(\A_n)$.
\end{theorem}
\begin{proof}
For each $n\in \bN$, the algebra $\A_n$ is naturally embedded in $\A$, and accordingly we let $E_n\in A$ denote the unit of $\A_n$. Given $a\in \A$, we then have that $E_n a E_n\in \A_n$ for each $n\in \bN$ and
\[
a=\oplus_{n=1}^\infty E_n a E_n=\lim_{N\to\infty}\oplus_{n=1}^N E_n a E_n
\]
where the limit exists in the norm topology.
Let 
\[
\mu:\A\to \rC^*_{\max}(\A) \qand \mu_n:\A_n\to \rC^*_{\max}(\A_n)
\]
denote the canonical embeddings. The map 
\[
\phi:\A\to \oplus_{n=1}^\infty \rC^*_{\max}(\A_n)
\]
defined as
\[
\phi(a)=\oplus_{n=1}^\infty \mu_n(E_n a E_n), \quad a\in \A
\]
is a completely contractive homomorphism. Thus, there is a $*$-homomorphism 
\[
\pi:\rC^*_{\max}(\A)\to \oplus_{n=1}^\infty \rC^*_{\max}(\A_n)
\]
with the property that $\pi\circ \mu=\phi$. On the other hand, for each $n\in \bN$ we define a completely contractive homomorphism 
\[
\psi_n:\A_n\to \rC^*_{\max}(\A)
\]
as
\[
\psi_n(a)=\mu(a), \quad a\in \A_n.
\]
Thus, there is a $*$-homomorphism 
\[
\sigma_n:\rC^*_{\max}(\A_n)\to \rC^*_{\max}(\A)
\]
such that $\sigma_n\circ \mu_n =\psi_n$. For each $n\in \bN$, put $P_n=\mu(E_n)$ which is a contractive idempotent, and thus a self-adjoint projection. Then, $\{P_n:n\in \bN\}$ is a collection of pairwise orthogonal  projections in $\rC^*_{\max}(\A)$. We obtain a $*$-homomorphism 
\[
\sigma:\oplus_{n=1}^\infty \rC^*_{\max}(\A_n)\to \rC^*_{\max} (\A)
\]
by setting
\[
\sigma(\oplus_{n=1}^\infty t_n)=\oplus_{n=1}^\infty P_n \sigma_n(t_n)P_n
\]
for every $\oplus_{n=1}^\infty t_n\in \oplus_{n=1}^\infty \rC^*_{\max}(\A_n)$. For $a\in \A$ we compute
\begin{align*}
(\sigma \circ\pi )(\mu(a))&=\sigma(\phi(a))=\sigma(\oplus_{n=1}^\infty \mu_n(E_na E_n))\\
&=\oplus_{n=1}^\infty P_n \sigma_n ( \mu_n(E_na E_n))P_n=\oplus_{n=1}^\infty P_n \psi_n(E_na E_n)P_n\\
&=\oplus_{n=1}^\infty P_n \mu(E_na E_n)P_n=\oplus_{n=1}^\infty  \mu(E_na E_n)\\
&=\lim_{N\to\infty} \mu\left( \oplus_{n=1}^N E_n a E_n\right)\\
&=\mu(a)
\end{align*}
where the limit exists in the norm topology. Next, for $\oplus_{n=1}^\infty \mu_n(a_n)\in \oplus_{n=1}^\infty \mu_n(\A_n)$ we find
\begin{align*}
(\pi\circ \sigma)(\oplus_{n=1}^\infty \mu_n(a_n))&=\pi\left( \oplus_{n=1}^\infty P_n\psi_n(a_n)P_n\right)\\
&=\pi\left( \oplus_{n=1}^\infty P_n \mu(a_n)P_n\right)=\pi\left( \oplus_{n=1}^\infty  \mu(E_n a_nE_n)\right)\\
&=\pi\left( \oplus_{n=1}^\infty  \mu( a_n)\right)=\lim_{N\to\infty} (\pi\circ \mu)(\oplus_{n=1}^N a_n)\\
&=\lim_{N\to\infty} \phi(\oplus_{n=1}^N a_n)=\phi(\oplus_{n=1}^\infty a_n)\\
&=\oplus_{n=1}^\infty \mu_n(a_n)
\end{align*}
where once again the limit exists in the norm topology. Thus, we conclude that $\pi$ is a $*$-isomorphism.
\end{proof}

One easy consequence goes as follows. 

\begin{corollary}\label{C:RFDdirectsum}
Let $\A$ be an operator algebra which can be written as $\A=\oplus_{n=1}^\infty \A_n$, where $\A_n$ is a unital finite-dimensional operator algebra for each $n\in \bN$. Then, $\rC^*_{\max}(\A)$ is RFD.
\end{corollary}
\begin{proof}
By Theorem \ref{T:C*maxfdimA}, we see that $\rC^*_{\max}(\A_n)$ is RFD for every $n\in \bN$, whence $\rC^*_{\max}(\A)$ is RFD by Theorem \ref{T:C*maxdirectsum}.
\end{proof}

This corollary says that for special RFD operator algebras $\A$, we can guarantee that $\rC^*_{\max}(\A)$ is RFD. The following example, inspired by multivariate operator theoretic considerations, supports the possibility that this may be a manifestation of a general phenomenon.

\begin{example}\label{E:ncdiscalg}
Let $d\geq 1$ be an integer and let $\bF_d^+$ denote the free semigroup on the generators $\{1,\ldots,d\}$. For each word $w\in \bF_d^+$, we let $\delta_w\in \ell^2(\bF_d^+)$ denote the characteristic function of $\{w\}$. Thus, $\{\delta_w:w\in \bF_d^+\}$ is an orthonormal basis of $\ell^2(\bF_d^+)$.
For each $1\leq k\leq d$, we define an isometry $L_k\in B(\ell^2(\bF_d^+))$ such that $L_k \delta_w=\delta_{kw}$ for every $w\in \bF_d^+$. We note that 
\[
\sum_{k=1}^d L_k L_k^*\leq I.
\] 
Then,  \emph{Popescu's disc algebra} $\fA_d\subset B(\ell^2(\bF_d^+))$ is the norm closed unital operator algebra generated by $L_1,\ldots,L_d$ \cite{popescu1991}. When $d=1$, the algebra $\fA_1$ can be identified with the classical disc algebra, consisting of those continuous functions on the closed complex unit disc that are holomorphic in the interior. 

We claim first that $\fA_d$ is RFD. To see this, note that $\fA_d$ can be identified with the tensor algebra of the $\rC^*$-correspondence $\bC^d$ over $\bC$ \cite[Example 2.7]{MS1998}. Thus, $\fA_d$ is RFD by virtue of Theorem \ref{T:fdcorrespondence}.
Next, we proceed to verify that $\rC^*_{\max}(\fA_d)$ is RFD. For that purpose, we will need the following important universality property of $\fA_d$. Let $\fH$ be a Hilbert space and let $T_1,\ldots,T_d\in B(\fH)$ be operators such that
\[
\sum_{k=1}^d T_kT_k^*\leq I. 
\]
In other words, the row operator $T=(T_1,\ldots,T_d):\fH^{(d)}\to \fH$ is contractive. Then, there is a unital completely contractive homomorphism $\Phi_T:\fA_d\to B(\fH)$ with the property that $\Phi_T(L_k)=T_k$ for every $1\leq k\leq d$ \cite[Theorem 2.1]{popescu1996}. Using that $(L_1,\ldots,L_d)$ is contractive, it is now straightforward to verify that $\rC^*_{\max}(\fA_d)$ is $*$-isomorphic to the free $\rC^*$-algebra of the operator algebra $\fA_d$. Thus, $\rC^*_{\max}(\fA_d)$ is RFD by \cite[Theorem 4.1]{pestov1994}.
\qed
\end{example}

Unfortunately, the general situation is more complicated and we briefly  indicate why by discussing two related settings.

When $d>1$, Popescu's disc algebra $\fA_d$ can be checked to be non-commutative. Notably, there is commutative version of it that acts on the symmetric Fock space over $\bC^d$ instead of the full one \cite{arveson1998},\cite{davidson1998}. It is typically denoted by $\A_d$, and it is easily checked to be RFD; more generally one could also invoke \cite[Example 5.2]{mittal2010}. Much as $\fA_d$ is universal for the so-called row contractions, the algebra $\A_d$ is universal for \emph{commuting} row contractions \cite[Theorem 6.2]{arveson1998}. This commutativity requirement prevents a direct adaptation of the previous argument, and in particular we do not know whether $\rC^*_{\max}(\A_d)$ is RFD when $d>1$.  

Another example is that of the bidisc algebra $A(\bD^2)$. As a uniform algebra,  $A(\bD^2)$ is clearly RFD as its characters completely norm it. Furthermore, a classical inequality due to Ando \cite{ando1963} shows that  $A(\bD^2)$ is universal for pairs of commuting contractions. Once again, this commutativity requirement complicates things and it does not appear to be known whether $\rC^*_{\max}(A(\bD^2))$ is RFD.

In view of these difficulties, we close this section by identifying a condition under which an RFD operator algebra $\A$ has the property that $\rC^*_{\max}(\A)$ is also RFD.  We start by recording an elementary fact.  

\begin{lemma}\label{L:projconv}
Let $\fH$ be a Hilbert space and let $(x_\lambda)_{\lambda\in \Lambda}$ be a net of contractions in $B(\fH)$. Let $\fM\subset \fH$ be a closed subspace which is invariant for  $\{x_\lambda:\lambda\in \Lambda\}$ and such that $\fM^\perp\subset \cap_{\lambda\in \Lambda}\ker x^*_\lambda$. Assume that $(x_\lambda P_{\fM})_{\lambda\in \Lambda}$ converges to $P_{\fM}$ in the weak operator topology. Then, $(x_\lambda)_{\lambda\in \Lambda}$ converges to $P_{\fM}$ in the weak operator topology.
\end{lemma}
\begin{proof}
By compactness of the closed unit ball in the weak operator topology, it suffices to show that $P_{\fM}$ is the only cluster point of $(x_\lambda)_{\lambda\in \Lambda}$ in the weak operator topology.  Let $x\in B(\fH)$ be such a cluster point. We know that $P_{\fM}x_{\lambda}P_{\fM}=x_{\lambda}P_{\fM}$ for every $\lambda\in\Lambda$. Using that  $(x_\lambda P_{\fM})_{\lambda\in \Lambda}$ converges to $P_{\fM}$ in the weak operator topology, we conclude that $P_{\fM}xP_{\fM}=P_{\fM}$. Moreover, $\fM^\perp\subset \ker x^*_\lambda$ so that $x^*_\lambda (I-P_{\fM})=0$ for every $\lambda\in \Lambda$, and thus $(I-P_{\fM})x=0$. Thus, we infer that
\[
x=P_{\fM}x=P_{\fM}xP_{\fM}+P_{\fM} x  (I-P_{\fM})=P_{\fM}+P_{\fM} x  (I-P_{\fM})
\]
and consequently
\[
xx^*=P_{\fM}+P_{\fM} x  (I-P_{\fM}) x^*P_\fM.
\]
On the other hand, since each $x_\lambda$ is a contraction we must have that $xx^*\leq I$ and therefore
\[
P_{\fM}+P_{\fM} x  (I-P_{\fM}) x^*P_\fM \leq P_{\fM}
\]
which forces $P_{\fM} x  (I-P_{\fM})=0$. Consequently, we have $x=P_{\fM}$. 
\end{proof}

Recall now that an operator algebra $\A$ is said to admit a \emph{contractive approximate identity} if there is a net of contractions $(e_\lambda)_{\lambda\in \Lambda}$ in $\A$ such that for every $a\in \A$, the nets $(e_\lambda a)_{\lambda\in \Lambda}$ and $(ae_\lambda)_{\lambda\in \Lambda}$ converge to $a$ in norm.

\begin{theorem}\label{T:idealRFD}
Let $\A$ be an operator algebra. For every $n\in\bN$, let $\J_n\subset \A$ be a closed two-sided ideal of $\A$ with a contractive approximate identity $(e^{(n)}_{\lambda})_{\lambda\in \Lambda_n}$. Assume that for every $a\in \A$ we have
\[
\lim_{n\to\infty}\liminf_{\lambda\in \Lambda_n} \|ae^{(n)}_{\lambda}\|=0 \qand \lim_{n\to\infty}\liminf_{\lambda\in \Lambda_n} \|e^{(n)}_{\lambda}a\|=0.
\]
Assume also that $\A/\J_n$ is finite-dimensional for every $n\in \bN$. 
Then, $\rC^*_{\max}(\A)$ is RFD.
\end{theorem}
\begin{proof}
The assumptions on $\A$ are invariant under completely isometric isomorphisms, so we may assume that $\A\subset \rC^*_{\max}(\A)\subset B(\fH)$ for some Hilbert space $\fH$. 
Since $\bC\langle \A,\A^*\rangle$ is dense in $\rC^*_{\max}(\A)$, it suffices to fix $s\in \bC\langle \A,\A^*\rangle$ with $\|s\|=1$ and $\eps>0$, and to find a finite-dimensional Hilbert space $\fN$ and a $*$-homomorphism $\pi:\rC^*_{\max}(\A)\to B(\fN)$  such that $\|\pi(s)\|\geq 1-\eps$.
Write
\[
s=\sum_{j=1}^r c^{(j)}_1 c^{(j)}_2\cdots c^{(j)}_{N_j}
\]
where $c^{(j)}_k\in \A\cup \A^*$ for $1\leq k\leq N_j, 1\leq j\leq r$. Choose a unit vectors $\xi\in \fH$ with the property that $\|s\xi\|\geq 1-\eps/2$.
For each $n\in \bN$, we put $\fM_n=\ol{\spn \J_n^*\fH}$.
It is readily verified that that $(e^{(n)*}_\lambda P_{\fM_n})_{\lambda\in \Lambda_n}$ converges in the strong operator topology to $P_{\fM_n}$. Furthermore, we see that
\[
\fM_n^\perp=\cap_{a\in \J_n}\ker a
\]
so by Lemma \ref{L:projconv} we see that $(e^{(n)*}_\lambda)_{\lambda\in \Lambda_n}$ converges in the weak operator topology to $P_{\fM_n}$. This implies that $(e^{(n)}_\lambda)_{\lambda\in \Lambda_n}$ converges in the weak operator topology to $P_{\fM_n}$ as well. Thus, for every $a\in \A$ we find
\[
\lim_{n\to\infty} \|aP_{\fM_n}\|\leq \lim_{n\to\infty} \liminf_{\lambda\in\Lambda_n}\|ae^{(n)}_\lambda\|=0
\]
and
\[
\lim_{n\to\infty}\|P_{\fM_n}a\|\leq\lim_{n\to\infty} \liminf_{\lambda\in\Lambda_n}\|e^{(n)}_\lambda a\|=0
\]
by assumption. 
We may thus choose a positive integer $m$ large enough so that
\[
\left\|s-  \sum_{j=1}^rP_{\fM_{m}^\perp}c^{(j)}_1P_{\fM_{m}^\perp} c^{(j)}_2P_{\fM_{m}^\perp }\cdots P_{\fM_{m}^\perp}c^{(j)}_{N_j}P_{\fM_{m}^\perp} \right\|< \eps/2
\]
whence
\[
\left\| \sum_{j=1}^rP_{\fM_{m}^\perp}c^{(j)}_1P_{\fM_{m}^\perp} c^{(j)}_2P_{\fM_{m}^\perp }\cdots P_{\fM_{m}^\perp}c^{(j)}_{N_j}P_{\fM_{m}^\perp}\xi  \right\|\geq 1-\eps.
\]
The map $\rho:\A\to B(\fM_m^\perp)$ defined as
\[
\rho(a)=a|_{\fM_m^\perp}, \quad a\in \A
\]
is a completely contractive homomorphism with  $\J_m\subset \ker \rho.$ By assumption we know that $\A/\J_m$ is finite-dimensional and thus so is $\rho(\A)$.
 Thus, by virtue of Lemma \ref{L:fdiminvsub}, there is a finite-dimensional subspace $\fN\subset \fM_m^\perp$ containing $P_{\fM^\perp_m}\xi$ that is invariant for $\rho(\A)$ and such that
\begin{align*}
& \sum_{j=1}^rP_{\fM_{m}^\perp}c^{(j)}_1P_{\fM_{m}^\perp} c^{(j)}_2P_{\fM_{m}^\perp }\cdots P_{\fM_{m}^\perp}c^{(j)}_{N_j}P_{\fM_{m}^\perp}\xi\\
&=\sum_{j=1}^r P_{\fN}c^{(j)}_1P_{\fN} c^{(j)}_2P_{\fN }\cdots P_{\fN}c^{(j)}_{N_j}P_{\fN}\xi.
\end{align*}
We infer that
\[
\left\| \sum_{j=1}^r P_{\fN}c^{(j)}_1P_{\fN} c^{(j)}_2P_{\fN }\cdots P_{\fN}c^{(j)}_{N_j}P_{\fN}\xi\right\|\geq 1-\eps.
\]
The map 
\[
a\mapsto a|_{\fN}, \quad a\in \A
\]
is a completely contractive homomorphism, so there is a $*$-homomorphism $\pi:\rC^*_{\max}(\A)\to B(\fN)$ such that $\pi(a)=a|_{\fN}$ for every $a\in \A$. We thus find
\[
\pi(s)= \sum_{j=1}^r P_{\fN}c^{(j)}_1P_{\fN} c^{(j)}_2P_{\fN }\cdots P_{\fN}c^{(j)}_{N_j}|_{\fN}
\]
so $\|\pi(s)P_{\fN}\xi\|\geq 1-\eps$. We conclude that $\|\pi(s)\|\geq 1-\eps$ and the proof is complete.

\end{proof}

It is a standard fact that every $\rC^*$-algebra admits a contractive approximate identity, but for general operator algebras the picture is more complicated (see \cite{effrosruan1990},\cite{BR2011},\cite{BR2013} for instance). At present we do not know if there are operator algebras satisfying the conditions of the previous theorem that are not of the form considered in Corollary \ref{C:RFDdirectsum}.

\section{Residual finite dimensionality of the $\rC^*$-envelope}\label{S:C*env}
In this section, we study the $\rC^*$-envelopes of RFD operator algebras. In some sense, our focus here is dual to that of the previous section. Indeed, in Section \ref{S:C*max} we considered the maximal $\rC^*$-cover of RFD operator algebras, whereas here we will study the minimal one. 

As in previous sections, we first analyze finite-dimensional operator algebras. In view of Corollary \ref{C:fdimRFD}, and Theorem \ref{T:C*maxfdimA}, a natural guess would be that these must necessarily have RFD $\rC^*$-envelopes. This is not the case as the next example shows.

\begin{example}\label{E:Paulsentrick}
Fix a positive integer $d\geq 2$ and let $H^2_d$ denote the Drury-Arveson space of holomorphic functions on the open unit ball $\bB_d\subset \bC^d$. Let $\S_d\subset B(H^2_d)$ denote the unital subspace generated by the operators $M_{z_1},\ldots,M_{z_d}$ of multiplication by the variables (the reader may consult  \cite{agler2002} for more detail about these objects). The $\rC^*$-algebra $\fT_d$ generated by $\S_d$ is called the \emph{Toeplitz algebra}, and we have $\fT_d\cong \rC^*_e(\S_d)$ \cite[Lemma 7.13 and Theorem 8.15]{arveson1998}. 
Next, let $\A_{\S_d}$ be the unital operator algebra constructed from $\S_d$ as in Lemma \ref{L:M2C*env}. Since $\S_d$ is finite-dimensional, so is $\A_{\S_d}$. Moreover, it follows from Lemma \ref{L:M2C*env} that 
\[
\rC^*_e(\A_{\S_d})\cong \bM_2(\rC^*_e(\S_d))\cong\bM_2(\fT_d).
\]
But since $\fT_d$ contains the ideal of compact operators \cite[Theorem 5.7]{arveson1998}, we infer that $\fT_d$ is not RFD, and thus $\bM_2(\fT_d)$ cannot be either.
\qed
\end{example}

Let $\A$ be a unital operator algebra. The embedding $\eps:\A\to \rC^*_e(\A)$ is a unital completely isometric homomorphism, so that a necessary condition for $\rC^*_e(\A)$ to be RFD is that $\A$ be such. Our basic goal in this section is to identify sufficient conditions for the converse to hold. For simplicity, we restrict our attention to the separable setting, but the interested reader will easily adapt the arguments to more general situations.

Let us now set some notation that we will use throughout. For each $n\in \bN$, let $r_n$ be some positive integer. The object of interest will be a unital operator algebra $\A\subset \prod_{n=1}^\infty\bM_{r_n}$. This unital completely isometric embedding will be fixed and part of the given data in our results. For each $m\in \bN$, we let $\gamma_m:\prod_{n=1}^\infty \bM_{r_n}\to \bM_{r_m}$ denote the natural projection, and we let $\fK$ denote the ideal of compact operators in $\rC^*(\A)$. Thus, 
\[
\fK=\{t\in \rC^*(\A):\lim_{n\to\infty}\|\gamma_n(a)\|=0\}.
\]
We let $\kappa:\rC^*(\A)\to \rC^*(\A)/\fK$ denote the quotient map.  An analysis of the representations of $\fK$ will provide insight into our problem. We first record an elementary fact.

\begin{lemma}\label{L:redK}
Let $\fA\subset \prod_{n=1}^\infty \bM_{r_n}$ be a $\rC^*$-algebra of compact operators. Let $E\subset \oplus_{n=1}^\infty\bC^{r_n}$ be a reducing subspace for $\fA$. If $\fA|_{E}$ is irreducible, then $E$ must be finite-dimensional.
\end{lemma}
\begin{proof}
Assume that $E$ is infinite-dimensional. For each $m\in \bN$, we let $P_m$ denote the orthogonal projection of $ \oplus_{n=1}^\infty\bC^{r_n}$ onto $\bC^{r_m}$. We claim that there is a vector $\xi\in E$ and a strictly increasing sequence of natural numbers $(N_j)_j$ with the property that $P_{N_j}\xi\neq 0$ for every $j\in \bN$. 

Choose $N_1\in \bN$ and a unit vector $\xi_1\in E$ such that $P_{N_1}\xi_1\neq 0$. Define $c_1=1$. Assume that for $m\in \bN$ we have constructed unit vector $\xi_1,\ldots,\xi_m\in E$, natural numbers $N_1<N_2<\ldots<N_m$ and positive numbers $c_1,\ldots,c_m$ such that
\begin{enumerate}

\item[\rm{(a)}]  $P_{N_j}\xi_j\neq 0$ for every $1\leq j\leq m$,

\item[\rm{(b)}]  $P_{N_k}\xi_j=0$  for every $1\leq j,k\leq m$ such that $k>j$,

\item[\rm{(c)}] $0<c_j<2^{-j}$  for every $1\leq j\leq m$ and
\[
c_j \|P_{N_k}\xi_j\|<\frac{1}{2^j}|c_k| \|P_{N_k}\xi_k\|
\]
for every $1\leq j,k\leq m$ such that $k<j$.
\end{enumerate}
Note that if  $\{n\in \bN:P_n \xi_k\neq 0\}$ is infinite for some $1\leq k\leq m$, then $\xi_k$ has the desired property and we are done. Without loss of generality, we may thus assume that the set $\{n\in \bN:P_n \xi_k\neq 0\}$ is finite for every $1\leq k\leq m$. By assumption, $E$ is infinite-dimensional so there is $N_{m+1}>N_m$ such that $P_{N_{m+1}}\xi_k=0$ for every $1\leq k\leq m$ and $P_{N_{m+1}}E\neq \{0\}$. Choose a unit vector $\xi_{m+1}\in E$ such that $P_{N_{m+1}}\xi_{m+1}\neq 0$. Choose also a real number $0<c_{m+1}<2^{-(m+1)}$ such that
\[
c_{m+1}\|P_{N_k}\xi_{m+1}\|<\frac{1}{2^{m+1}}|c_k| \|P_{N_k}\xi_k\|
\]
for every $1\leq k\leq m$. By induction, we obtain a sequence $(\xi_j)_j$ of unit vectors in $E$, a strictly increasing sequence of natural numbers $(N_j)_j$ and a sequence of positive numbers $(c_j)_j$ with the property that
\begin{enumerate}

\item[\rm{(a')}]  $P_{N_j}\xi_j\neq 0$ for every $j\in \bN$,

\item[\rm{(b')}]  $P_{N_k}\xi_j=0$  for every $j\in \bN$ and every $k>j$,

\item[\rm{(c')}] $0<c_j<2^{-j}$  for every $j\in\bN$ and
\[
c_j \|P_{N_k}\xi_j\|<\frac{1}{2^j}|c_k| \|P_{N_k}\xi_k\|
\]
for every $j\in \bN$ and every $k<j$.
\end{enumerate}
Put $\xi=\sum_{j=1}^\infty c_j \xi_j$. Then, using (c') we see that
\[
\|\xi\|\leq \sum_{j=1}^\infty \frac{ \|\xi_j\|}{2^j}=1
\]
and $\xi\in E$. For $k\in \bN$, we compute using (a'),(b') and (c') that
\begin{align*}
\|P_{N_k}\xi\|&\geq \left\|P_{N_k}\left(\sum_{j=1}^k c_j\xi_j\right) \right\|-\left \|P_{N_k}\left(\sum_{j=k+1}^\infty c_j\xi_j\right) \right\|\\
&\geq c_k \|P_{N_k}\xi_k\|-\sum_{j=k+1}^\infty c_j \|P_{N_k}\xi_j\|\\
&\geq c_k \|P_{N_k}\xi_k\|\left(1-\sum_{j=k+1}^\infty \frac{1}{2^j}\right)\geq \frac{c_k}{2} \|P_{N_k}\xi_k\|>0.
\end{align*}
The claim is established.

Now, we know that $\fA|_E$ is an irreducible $\rC^*$-algebra of compact operators, and thus it must consist of all compact operators on $E$. In particular, there is $\Xi\in \fA$ with the property that $\Xi|_E=\xi\otimes \xi$. On the other hand, we have $\Xi\in \fA$ so we can write $\Xi=\oplus_{n=1}^\infty \Xi_n$ where $\Xi_n:\bC^{r_n}\to\bC^{r_n}$ for every $n\in \bN$. Note that  $\Xi \xi=\xi$, so that $\Xi_n P_{n}\xi=P_n \xi$ for every $n\in \bN$. Using that $P_{N_j}\xi\neq 0$ for every $j\in \bN$, we see that $\|\Xi_{N_j}\|\geq 1$ for every $j\in \bN$, which contradicts the fact that $\Xi$ is compact.
\end{proof}

Next, we leverage this fact to identify certain representations that preserve the residual finite-dimensionality of a $\rC^*$-algebra.

\begin{lemma}\label{L:repfK}
Let $\fA\subset \prod_{n=1}^\infty \bM_{r_n}$ be a unital $\rC^*$-algebra and let $\fL$ denote the ideal of compact operators in $\fA$. Let $\pi:\fA\to B(\fH)$ be a unital $*$-homomorphism with the property that $\pi(\fL)$ is non-degenerate. Then, there is a collection $\{E_\lambda\}_{\lambda\in \Lambda}$ of finite-dimensional reducing subspaces for $\fA$ with the property that $\fA|_{E_\lambda}$ is irreducible for every $\lambda\in \Lambda$, and such that there is a unitary $U:\fH\to \oplus_{\lambda\in \Lambda}E_\lambda$ satisfying
\[
U\pi(s)U^*=\oplus_{\lambda\in \Lambda}(s|_{E_\lambda}), \quad s\in \fA.
\]
\end{lemma}
\begin{proof}
By Theorem \cite[Theorem 1.4.4]{arveson1976inv}, we see that there is a collection $\{E_\lambda\}_{\lambda\in \Lambda}$ of reducing  subspaces for $\fL$ with the property that $\{x|_{E_\lambda}:x\in\fL\}$ is irreducible for every $\lambda\in \Lambda$, along with a unitary $U:\fH\to \oplus_{\lambda\in \Lambda}E_\lambda$ such that
\[
U\pi(x)U^*=\oplus_{\lambda\in \Lambda} (x|_{E_\lambda}), \quad x\in \fL.
\]
By virtue of Lemma \ref{L:redK}, we see that $E_\lambda$ is finite-dimensional for every $\lambda\in \Lambda$. Now, we note that $E_\lambda=\spn \fL E_\lambda$, whence $E_\lambda$ is reducing for $\fA$ for every $\lambda\in \Lambda$ and a standard verification reveals that
\[
U\pi(s)U^*=\oplus_{\lambda\in \Lambda} (s|_{E_\lambda}), \quad s\in \fA.
\]
\end{proof}

We can now exhibit a sufficient condition for the $\rC^*$-envelope of an RFD operator algebra to be RFD as well.

\begin{theorem}\label{T:quotientRFD}
Let $\A\subset \prod_{n=1}^\infty\bM_{r_n}$ be a unital operator algebra. Assume that every $\rC^*$-algebra which is a quotient of $\rC^*(\A)/\fK$ is RFD. Then, $\rC^*_e(\A)$ is RFD.
\end{theorem}
\begin{proof}
Let $\pi:\rC^*(\A)\to B(\fH)$ be a unital $*$-homomorphism that is completely isometric on $\A$ and that has the unique extension property with respect to $\A$ (such things exist by \cite{dritschel2005}). Then, $\rC^*_e(\A)\cong \pi(\rC^*(\A))$. Basic representation theory for $\rC^*$-algebras (see the discussion preceding \cite[Theorem I.3.4]{arveson1976inv}) dictate that we may decompose $\pi$ as $\pi=\pi_\fK\oplus \sigma$, where $\pi_\fK(\fK)$ is non-degenerate and $\fK\subset \ker\sigma$. We know that $\sigma(\rC^*(\A))$ is $*$-isomorphic to a quotient of $\rC^*(\A)/\fK$, and hence is RFD by assumption. But so is $\pi_\fK(\rC^*(\A))$ by Lemma \ref{L:repfK}. Hence, 
\[
\rC^*_e(\A)\cong \pi(\rC^*(\A))\subset \pi_\fK(\rC^*(\A))\oplus \sigma(\rC^*(\A))
\]
is RFD.
\end{proof}

The following consequence is noteworthy.

\begin{corollary}\label{C:quotientRFD}
Let $\A\subset \prod_{n=1}^\infty\bM_{r_n}$ be a unital operator algebra. Then, $\rC^*_e(\A)$ is RFD if one of the following conditions holds:
\begin{enumerate}

\item[\rm{(a)}] $\rC^*(\A)/\fK$ is commutative,

\item[\rm{(b)}] $\rC^*(\A)/\fK$ is finite-dimensional.
\end{enumerate}
\end{corollary}
\begin{proof}
Both assumptions are readily seen to imply that every $\rC^*$-algebra which is a  quotient of $\rC^*(\A)/\fK$ is RFD, so the result follows at once from Theorem \ref{T:quotientRFD}.
\end{proof}

We now identify another context where the $\rC^*$-envelope can be shown to be RFD. The next result is reminiscent of Arveson's boundary theorem \cite[Theorem 2.1.1]{arveson1972}. 

\begin{theorem}\label{T:smallkappa}
Let $\A\subset \prod_{n=1}^\infty\bM_{r_n}$ be a unital operator algebra. Assume that for every $d\in \bN$, there is a dense subset $\D_d\subset \bM_d(\A)$ with the property that $\|\kappa^{(d)}(A)\|<\|A\|$ for every $A\in \D_d$. Then, there is a collection $\{E_\lambda\}_{\lambda\in \Lambda}$ of finite-dimensional reducing subspaces for $\rC^*(\A)$  such that the unital $*$-homomorphism 
\[
s\mapsto \oplus_{\lambda\in \Lambda}(s|_{E_\lambda}), \quad s\in \rC^*(\A)
\]
is completely isometric on $\A$ and has the unique extension property with respect to $\A$. Moreover, $\rC^*(\A)|_{E_{\lambda}}$ is irreducible for every $\lambda\in \Lambda.$  In particular
 $\rC^*_e(\A)$ is $*$-isomorphic to
\[
\{\oplus_{\lambda\in \Lambda}(s|_{E_\lambda}):s\in \rC^*(\A)\}
\]
and it is RFD.
\end{theorem}
\begin{proof}
Let $\pi:\rC^*(\A)\to B(\fH)$ be a unital $*$-homomorphism that is completely isometric on $\A$ and that has the unique extension property with respect to $\A$. As before, we may decompose $\pi$ as $\pi_\fK\oplus \sigma$, where $\pi_\fK(\fK)$ is non-degenerate and $\fK\subset \ker \sigma$. In particular, we see that
\[
\|\sigma^{(d)}(A)\|\leq \|\kappa^{(d)}(A)\|<\|A\|
\]
for every $A\in \D_d$ and every $d\in\bN$. The fact that $\pi$ is completely isometric on $\A$ implies 
\[
\|\pi_{\fK}^{(d)}(A)\|=\|A\|
\]
for every $A\in \D_d$ and every $d\in \bN$. Since $\D_d$ is dense in $\bM_d(\A)$ for every $d\in \bN$, we infer that $\pi_{\fK}$ is completely isometric on $\A$.  Moreover, it is easily verified that $\pi_{\fK}$ inherits from $\pi$ the unique extension property with respect to $\A$. The conclusion now follows from Lemma \ref{L:repfK} applied to $\pi_{\fK}$.
\end{proof}

Next, we show that if the algebra $\A$ contains many compact operators, then the condition of Theorem \ref{T:smallkappa} is automatically satisfied. Recall that if $\A$ and $\B$ are operator algebras, then a completely contractive surjective homomorphism $\pi:\A\to\B$ is a \emph{complete quotient map} if the induced map $\widehat \pi:\A/\ker \pi\to \B$ is a complete isometry. 

\begin{theorem}\label{T:epssurj}
Let $\A\subset \prod_{n=1}^\infty \bM_{r_n}$ be a unital operator algebra. Assume that there is $N\in \bN$ with the property that $\gamma_n|_{\fK\cap \A}$ is a complete quotient map onto $\bM_{r_n}$ for every $n\geq N$.  Then, there is a collection $\{E_\lambda\}_{\lambda\in \Lambda}$ of finite-dimensional reducing subspaces for $\rC^*(\A)$  such that the unital $*$-homomorphism 
\[
s\mapsto \oplus_{\lambda\in \Lambda}(s|_{E_\lambda}), \quad s\in \rC^*(\A)
\]
is completely isometric on $\A$ and has the unique extension property with respect to $\A$. Moreover, $\rC^*(\A)|_{E_{\lambda}}$ is irreducible for every $\lambda\in \Lambda.$  In particular
 $\rC^*_e(\A)$ is $*$-isomorphic to
\[
\{\oplus_{\lambda\in \Lambda}(s|_{E_\lambda}):s\in \rC^*(\A)\}
\]
and it is RFD.
\end{theorem}
\begin{proof}
Let $d\in \bN$ and let $A\in \bM_d(\A)$ be such that $\|\kappa^{(d)}(A)\|=\|A\|=1$. By virtue of Lemma \ref{L:limsup}, we conclude that 
\[
\limsup_{n\to\infty}\|\gamma_n^{(d)}(A)\|=\|A\|.
\]
Hence, we have that $\|A\|=\sup_{n\geq N}\|\gamma_n^{(d)}(A)\|$.  Let $0<\delta<1$ and choose $n\geq N$ such that 
\[
(1-\delta/2)\|A\|\leq \|\gamma_n^{(d)}(A)\|.
\]
Next, we note that $\gamma_n^{(d)}(A)\in \bM_d(\bM_{r_n})$ so there is a unit vector $\xi\in (\bC^{r_n})^{(d)}$ with the property that 
\[
\|\gamma_n^{(d)}(A)\xi\|=\|\gamma_n^{(d)}(A)\|.
\]
Let $R\in \bM_d(\bM_{r_n})$ denote the rank-one contraction such that $R\xi=\gamma_n^{(d)}(A)\xi$.
Since $\gamma_n|_{\fK\cap \A}$ is a complete quotient map, we may find $K\in \bM_d(\fK\cap \A)$ with $\|K\|\leq 2$ and such that $\gamma_n^{(d)}(K)=R.$ Now, we calculate
\begin{align*}
\|A+\delta K\|&\geq \|\gamma_n^{(d)}(A+\delta K) \|=\| \gamma_n^{(d)}(A)+\delta R\|\\
&\geq \|\gamma_n^{(d)}(A)\xi+\delta R\xi\|=(1+\delta)\|\gamma_n^{(d)}(A)\xi\|\\
&=(1+\delta)\|\gamma_n^{(d)}(A)\|\geq (1+\delta)(1-\delta/2)\|A\|\\
&>\|A\|.
\end{align*}
In particular, we find
\[
\|\kappa^{(d)}(A+\delta K)\|=\|\kappa^{(d)}(A)\|=\|A\|<\|A+\delta K\|.
\]
Noting that $\|A-(A+\delta K)\|\leq 2\delta$, we may invoke Theorem \ref{T:smallkappa} to obtain the desired conclusions.
\end{proof}

The reader will glean from the proof that the assumption on $\gamma_n|_{\fK\cap \A}$ being a complete quotient map for every $n\geq N$ can be weakened. It suffices to require that the sequence of inverses of the induced maps on the quotients $(\A\cap \fK)/\ker (\gamma_n|_{\A\cap \fK})$ be uniformly completely bounded.

As an application of Theorem \ref{T:epssurj} we single out the following concrete consequence.

\begin{corollary}\label{C:AcontainsK}
Let $\A\subset \prod_{n=1}^\infty \bM_{r_n}$ be a unital operator algebra which contains $\oplus_{n=1}^\infty \bM_{r_n}$. Then, $\rC^*_e(\A)\cong \rC^*(\A)$ and in particular $\rC^*_e(\A)$ is RFD.
\end{corollary}
\begin{proof}
The assumption that $\A$ contains $\oplus_{n=1}^\infty \bM_{r_n}$ is easily seen to imply that $\gamma_n|_{ \fK\cap \A}$ is a complete quotient map for every $n\in \bN$. By Theorem \ref{T:epssurj}, we see that there is a collection $\{E_\lambda\}_{\lambda\in \Lambda}$ of finite-dimensional reducing subspaces for $\rC^*(\A)$  such that the unital $*$-homomorphism $\pi:\rC^*(\A)\to \prod_{\lambda\in \Lambda} B(E_\lambda)$ defined as
\[
\pi(s)= \oplus_{\lambda\in \Lambda}(s|_{E_\lambda}), \quad s\in \rC^*(\A)
\]
is completely isometric on $\A$ and has the unique extension property with respect to $\A$. 
 Moreover, $\rC^*(\A)|_{E_{\lambda}}$ is irreducible for every $\lambda\in \Lambda.$  We may assume without loss of generality that the the subspaces $\{ E_\lambda :\lambda\in \Lambda\}$ are distinct. We claim that $\pi$ is a $*$-isomorphism.

To see thus, let $m\in \bN$ and let $p_m\in \oplus_{n=1}^\infty \bM_{r_n}$ be the orthogonal projection onto $\bC^{r_m}$. By assumption, we see that $p_m\in \A$ so that 
\[
\oplus_{\lambda\in \Lambda}(p_m|_{E_\lambda})=\pi(p_m)\neq 0.
\]
Let $\lambda\in \Lambda$. The subspace $E_\lambda$ is reducing for $\rC^*(\A)$, and in particular for $\oplus_{n=1}^\infty \bM_{r_n}$. Using that $\rC^*(\A)|_{E_{\lambda}}$ is irreducible, we see that $E_\lambda$  must coincide with one of the orthogonal summands $\bC^{r_{n_\lambda}}\subset \oplus_{n=1}^\infty \bC^{r_n}$.  The fact that
\[
\oplus_{\lambda\in \Lambda}(p_m|_{E_\lambda})\neq 0
\]
for every $m\in \bN$ shows that
\[
\oplus_{\lambda\in \Lambda}(s|_{E_\lambda}) \qand \oplus_{n=1}^\infty \gamma_n(s)
\]
 coincide for every $s\in \rC^*(\A)$, up to a fixed unitary permutation of the summands. Thus, $\pi$ is a $*$-isomorphism and we find
\[
\rC^*_e(\A)\cong \pi(\rC^*(\A))\cong \rC^*(\A).
\]
\end{proof}

We now exhibit an example of an operator algebra that satisfies the condition of Theorem \ref{T:smallkappa} but does not satisfy those of Corollary \ref{C:quotientRFD} or of Theorem \ref{T:epssurj}.

\begin{example}
For $n\in \bN$ and $1\leq i,j\leq n$, we let $E_{i,j}^{(n)}\in \bM_n$ denote the standard matrix unit.
Given $m\in \bN$, we let $T_m\in \prod_{n=1}^\infty \bM_n$ be the unique element satisfying
\[
\gamma_n(T_m)=\begin{cases}
E^{(n)}_{2m-1,2m}  & \text{ if } n=2m,\\
\frac{1}{2}E^{(n)}_{2m-1,2m}  & \text{ if } n>2m,\\
0 & \text{ otherwise}.
\end{cases}
\] 
Let $\A\subset\prod_{n=1}^\infty \bM_n$ be the unital operator algebra generated by $\{T_m:m\in \bN\}$. For every $n>2$, we note that
\[
\gamma_n(T_1^*T_1-T_1T_1^*)=\frac{1}{4}(E_{22}^{(n)}-E_{11}^{(n)}).
\]
By Lemma \ref{L:limsup} we find
\[
\|\kappa(T_1^*T_1-T_1T_1^*)\|=\limsup_{n\to\infty}\|\gamma_n(T_1^*T_1-T_1T_1^*)\|=\frac{1}{4}
\]
which shows that $\rC^*(\A)/\fK$ is not commutative. Next, assume that there is $r\in \bN$ along with $\alpha_1,\ldots,\alpha_r\in \bC$ such that
\[
\sum_{j=1}^r \alpha_j \kappa(T_j)=0.
\]
Note now that for $n>2r$ we have
\[
\gamma_n\left( \sum_{j=1}^r \alpha_j T_j\right)=\frac{1}{2}\sum_{j=1}^r \alpha_j E^{(n)}_{2j-1,2j}
\]
whence
\[
\left\|\gamma_n\left(\sum_{j=1}^r \alpha_j T_j \right)\right\|\geq \frac{1}{2}\max_{1\leq j\leq r}|\alpha_j|.
\]
By virtue of Lemma \ref{L:limsup} again, we see that
\[
 \frac{1}{2}\max_{1\leq j\leq r}|\alpha_j|\leq \limsup_{n\to\infty}\left\|\gamma_n\left(\sum_{j=1}^r \alpha_j T_j \right)\right\|=\left\| \sum_{j=1}^r \alpha_j \kappa(T_j)\right\|=0
\]
so that $\alpha_1=\alpha_2=\ldots=\alpha_r=0$. We conclude that the set $\{\kappa(T_j):j\in \bN\}$ is linearly independent in $\rC^*(\A)/\fK$, whence $\rC^*(\A)/\fK$ is infinite-dimensional. This shows that $\A$ does not satisfy either condition in Corollary \ref{C:quotientRFD}.

Fix $d\in \bN$. It is readily seen that $T_iT_j=0$ for every $i,j\in \bN$, so that a generic element $A\in \bM_d(\A)$ can be written as
\[
A=C_0 \otimes I+C_1\otimes T_1+\ldots+ C_r \otimes T_r
\]
for some $r\in \bN$ and some $C_0,\ldots,C_r\in \bM_d$. Here, given $C=[c_{ij}]_{i,j}\in \bM_d$ and $a\in \A$, we use the notation
\[
C\otimes a=[c_{ij}a]_{i,j}\in \bM_d(\A).
\]
For convenience, for each $1\leq k\leq r$ we let
\[
\Gamma_k=\begin{bmatrix}
C_0 & C_k \\
0 & C_0
\end{bmatrix}
\qand
\Gamma'_k=\begin{bmatrix}
C_0 & \frac{1}{2}C_k \\
0 &C_0
\end{bmatrix}.
\]
Upon applying the canonical shuffle in $\bM_d(\bM_n)$,  we find that if $n=2p$ for some $1\leq p\leq r$ then
\[
\gamma^{(d)}_{n}(A)=
\Gamma_1'\oplus \Gamma'_2\oplus \ldots \oplus \Gamma'_{p-1}\oplus  \Gamma_p.
\]
Likewise, if $n>2r$ we find 
\[
\gamma^{(d)}_{n}(A)=
\Gamma_1'\oplus \Gamma'_2\oplus \ldots \oplus \Gamma'_r\oplus C_0 I_{n-2r}.
\]
In particular, we see that $\kappa^{(d)}(A)\neq 0$ unless $A=0$, which implies that $\A\cap \fK=\{0\}$ and thus $\A$ does not satisfy the condition of Theorem \ref{T:epssurj}.

Finally, we show that $\A$ satisfies the condition of Theorem \ref{T:smallkappa}. We let $\D_d\subset \bM_n(\A)$ be the subset consisting of elements of the form
\[
A=C_0 \otimes I+C_1 \otimes T_1+\ldots+ C_r\otimes T_r 
\]
for some $r\in \bN$ and some \emph{invertible} matrices $C_0,\ldots,C_r\in \bM_d$. It is clear that $\D_d$ is dense in $\bM_d(\A)$. We now claim that $\|\kappa^{(d)}(A)\|<\|A\|$ for every $A\in \D_d$. To see this, fix an element $A\in \D_d$ which we write as
\[
A=C_0 \otimes I+C_1 \otimes T_1+\ldots+ C_r\otimes T_r
\]
for some $r\in \bN$ and some invertible matrices $C_0,\ldots,C_r\in \bM_d$. By Lemma \ref{L:limsup}, we must show that
\[
\limsup_{n\to\infty}\|\gamma^{(d)}_{n}(A)\|<\|A\|.
\]
Using the same notation as before, it is clear that we have that $\|\Gamma'_k\|\geq \|C_0\|$ for every $1\leq k\leq r$, so we have
\[
\limsup_{n\to\infty}\|\gamma^{(d)}_{n}(A)\|=\max_{1\leq k\leq r}\|\Gamma'_k\|.
\]
On the other hand, for each $1\leq k\leq r$, since $C_k$ is invertible there is $\delta_k>0$ such that $C_kC_k^*\geq \delta_k I$, whence
\begin{align*}
\left\|\begin{bmatrix}
C_0 & C_k\\
0 & C_0
\end{bmatrix} \right\|^2&\geq \left\|\begin{bmatrix}
C_0 & C_k
\end{bmatrix} \right\|^2\\
&=\|C_0C_0^*+C_kC_k^*\|\\
&\geq \|C_0\|^2+\delta_k>\|C_0\|^2
\end{align*}
and
\begin{align*}
\|\Gamma_k'\|=\left\|\begin{bmatrix}
C_0 & \frac{1}{2}C_k\\
0 & C_0
\end{bmatrix} \right\|&\leq \frac{1}{2}\left\|\begin{bmatrix}
C_0 & C_k\\
0 & C_0
\end{bmatrix} \right\|+\frac{1}{2}\left\|\begin{bmatrix}
C_0 &0\\
0 & C_0
\end{bmatrix} \right\|\\
&=\frac{1}{2}\left\|\begin{bmatrix}
C_0 & C_k\\
0 & C_0
\end{bmatrix} \right\|+\frac{1}{2}\|C_0\|\\
&<\left\|\begin{bmatrix}
C_0 & C_k\\
0 & C_0
\end{bmatrix} \right\|=\|\Gamma_k\|
\end{align*}
for every $1\leq k\leq r$. We conclude that
\begin{align*}
\limsup_{n\to\infty}\|\gamma^{(d)}_{n}(A)\|&=\max_{1\leq k\leq r}\|\Gamma'_k\|<\max_{1\leq k\leq r}\|\Gamma_k\|\\
&\leq \max_{1\leq p\leq r}\|\gamma_{2p}^{(d)}(A)\|\leq \|A\|.
\end{align*}
Hence, $\rC^*_e(\A)$ is RFD.
\qed
\end{example}

Finally, we provide an example that shows that residual finite-dimensionality of both the minimal and maximal $\rC^*$-cover of a unital operator algebra does not typically imply that the same property holds for arbitrary $\rC^*$-covers.

\begin{example}\label{E:RFDcovers}
Let $H^2(\bD)$ denote the Hardy space on the open unit disc $\bD\subset \bC$ and let $S\in B(H^2(\bD))$ denote the usual isometric unilateral shift (see \cite{agler2002} for details). Let $\S=\spn\{I,S,\}\subset B(H^2(\bD))$. It is well known that $\rC^*(\S)$ contains the ideal $\fK$ of compact operators on $H^2(\bD)$, so in particular $\rC^*(\S)$ is not RFD. Moreover, the quotient $\rC^*(\S)/\fK$ is $*$-isomorphic to $\rC(\bT)$. In particular, this implies that $\rC^*_e(\S)\cong \rC(X)$ where $X\subset \bT$ is the Shilov boundary of $\spn\{1,z\}$. As seen in Example \ref{E:fdoa}, $X=\bT$ and thus $\rC^*_e(\S)\cong \rC(\bT)$.
Consider now the unital operator algebra $\A_\S$ from Lemma \ref{L:M2C*env}, which is finite-dimensional because $\S$ is. Thus, $\rC^*_{\max}(\A_\S)$ is RFD by Theorem \ref{T:C*maxfdimA}. Furthermore, by Lemma \ref{L:M2C*env} we know that
\[
\rC_e^*(\A_\S)\cong \bM_2(\rC^*_e(\S))\cong \bM_2(\rC(\bT))
\]
which is also RFD. 
\qed
\end{example}

\bibliography{/Users/raphaelclouatre/Dropbox/Research/Shared/Chris-Raphael/RFDenvelope/biblio_main_RFD}

\def\cprime{$'$}
\begin{bibdiv}
\begin{biblist}

\bib{agler2002}{book}{
      author={Agler, Jim},
      author={McCarthy, John~E.},
       title={Pick interpolation and {H}ilbert function spaces},
      series={Graduate Studies in Mathematics},
   publisher={American Mathematical Society, Providence, RI},
        date={2002},
      volume={44},
        ISBN={0-8218-2898-3},
         url={http://dx.doi.org.proxy.lib.uwaterloo.ca/10.1090/gsm/044},
      review={\MR{1882259 (2003b:47001)}},
}

\bib{ando1963}{article}{
      author={And\^o, T.},
       title={On a pair of commutative contractions},
        date={1963},
        ISSN={0001-6969},
     journal={Acta Sci. Math. (Szeged)},
      volume={24},
       pages={88\ndash 90},
      review={\MR{0155193}},
}

\bib{archbold1995}{article}{
      author={Archbold, R.~J.},
       title={On residually finite-dimensional {$C^*$}-algebras},
        date={1995},
        ISSN={0002-9939},
     journal={Proc. Amer. Math. Soc.},
      volume={123},
      number={9},
       pages={2935\ndash 2937},
         url={https://doi-org.uml.idm.oclc.org/10.2307/2160599},
      review={\MR{1301006}},
}

\bib{arveson1969}{article}{
      author={Arveson, William},
       title={Subalgebras of {$C^{\ast} $}-algebras},
        date={1969},
        ISSN={0001-5962},
     journal={Acta Math.},
      volume={123},
       pages={141\ndash 224},
      review={\MR{0253059 (40 \#6274)}},
}

\bib{arveson1972}{article}{
      author={Arveson, William},
       title={Subalgebras of {$C^{\ast} $}-algebras. {II}},
        date={1972},
        ISSN={0001-5962},
     journal={Acta Math.},
      volume={128},
      number={3-4},
       pages={271\ndash 308},
      review={\MR{0394232 (52 \#15035)}},
}

\bib{arveson1976inv}{book}{
      author={Arveson, William},
       title={An invitation to {$C\sp*$}-algebras},
   publisher={Springer-Verlag, New York-Heidelberg},
        date={1976},
        note={Graduate Texts in Mathematics, No. 39},
      review={\MR{0512360}},
}

\bib{arveson1998}{article}{
      author={Arveson, William},
       title={Subalgebras of {$C^*$}-algebras. {III}. {M}ultivariable operator
  theory},
        date={1998},
        ISSN={0001-5962},
     journal={Acta Math.},
      volume={181},
      number={2},
       pages={159\ndash 228},
         url={http://dx.doi.org/10.1007/BF02392585},
      review={\MR{1668582 (2000e:47013)}},
}

\bib{arveson2008}{article}{
      author={Arveson, William},
       title={The noncommutative {C}hoquet boundary},
        date={2008},
        ISSN={0894-0347},
     journal={J. Amer. Math. Soc.},
      volume={21},
      number={4},
       pages={1065\ndash 1084},
         url={http://dx.doi.org/10.1090/S0894-0347-07-00570-X},
      review={\MR{2425180 (2009g:46108)}},
}

\bib{blackadar2006}{book}{
      author={Blackadar, B.},
       title={Operator algebras},
      series={Encyclopaedia of Mathematical Sciences},
   publisher={Springer-Verlag, Berlin},
        date={2006},
      volume={122},
        ISBN={978-3-540-28486-4; 3-540-28486-9},
         url={https://doi-org.uml.idm.oclc.org/10.1007/3-540-28517-2},
        note={Theory of $C^*$-algebras and von Neumann algebras, Operator
  Algebras and Non-commutative Geometry, III},
      review={\MR{2188261}},
}

\bib{blecher1999}{article}{
      author={Blecher, David~P.},
       title={Modules over operator algebras, and the maximal
  {$C^\ast$}-dilation},
        date={1999},
        ISSN={0022-1236},
     journal={J. Funct. Anal.},
      volume={169},
      number={1},
       pages={251\ndash 288},
         url={https://doi-org.uml.idm.oclc.org/10.1006/jfan.1999.3491},
      review={\MR{1726755}},
}

\bib{BR2011}{article}{
      author={Blecher, David~P.},
      author={Read, Charles~John},
       title={Operator algebras with contractive approximate identities},
        date={2011},
        ISSN={0022-1236},
     journal={J. Funct. Anal.},
      volume={261},
      number={1},
       pages={188\ndash 217},
         url={http://dx.doi.org.uml.idm.oclc.org/10.1016/j.jfa.2011.02.019},
      review={\MR{2785898}},
}

\bib{BR2013}{article}{
      author={Blecher, David~P.},
      author={Read, Charles~John},
       title={Operator algebras with contractive approximate identities, {II}},
        date={2013},
        ISSN={0022-1236},
     journal={J. Funct. Anal.},
      volume={264},
      number={4},
       pages={1049\ndash 1067},
         url={http://dx.doi.org.uml.idm.oclc.org/10.1016/j.jfa.2012.11.013},
      review={\MR{3004957}},
}

\bib{BO2008}{book}{
      author={Brown, N.P},
      author={Ozawa, N.},
       title={${C}^*$-algebras and finite-dimensional approximations},
      series={Graduate Studies in Mathematics},
   publisher={American Mathematical Society},
     address={Providence, RI},
        date={2008},
      volume={88},
}

\bib{choi1980}{article}{
      author={Choi, Man~Duen},
       title={The full {$C^{\ast} $}-algebra of the free group on two
  generators},
        date={1980},
        ISSN={0030-8730},
     journal={Pacific J. Math.},
      volume={87},
      number={1},
       pages={41\ndash 48},
         url={http://projecteuclid.org.uml.idm.oclc.org/euclid.pjm/1102780313},
      review={\MR{590864}},
}

\bib{clouatre2015CB}{article}{
      author={Clou{\^a}tre, Rapha{\"e}l},
       title={Completely bounded isomorphisms of operator algebras and
  similarity to complete isometries},
        date={2015},
        ISSN={0022-2518},
     journal={Indiana Univ. Math. J.},
      volume={64},
      number={3},
       pages={825\ndash 846},
         url={http://dx.doi.org.uml.idm.oclc.org/10.1512/iumj.2015.64.5572},
      review={\MR{3361288}},
}

\bib{CM2017rfd}{article}{
      author={Clou\^atre, Rapha\"el},
      author={Marcoux, Laurent~W.},
       title={Residual finite dimensionality and representations of amenable
  operator algebras},
        date={2017},
     journal={preprint arXiv:1612.01050},
}

\bib{courtney2017}{article}{
      author={Courtney, Kristin},
      author={Shulman, Tatiana},
       title={Elements of ${C}^*$-algebras attaining their norm in a
  finite-dimensional representation},
        date={2017},
     journal={to appear in Canad. J. Math.},
}

\bib{DFK2017survey}{article}{
      author={Davidson, Kenneth~R.},
      author={Fuller, Adam~H.},
      author={Kakariadis, Evgenios T.~A.},
       title={Semicrossed products of operator algebras: A survey},
     journal={New York J. Math., to appear.},
}

\bib{davidsonkennedy2015}{article}{
      author={Davidson, Kenneth~R.},
      author={Kennedy, Matthew},
       title={The {C}hoquet boundary of an operator system},
        date={2015},
        ISSN={0012-7094},
     journal={Duke Math. J.},
      volume={164},
      number={15},
       pages={2989\ndash 3004},
         url={http://dx.doi.org.uml.idm.oclc.org/10.1215/00127094-3165004},
      review={\MR{3430455}},
}

\bib{davidson1998}{article}{
      author={Davidson, Kenneth~R.},
      author={Pitts, David~R.},
       title={Nevanlinna-{P}ick interpolation for non-commutative analytic
  {T}oeplitz algebras},
        date={1998},
        ISSN={0378-620X},
     journal={Integral Equations Operator Theory},
      volume={31},
      number={3},
       pages={321\ndash 337},
         url={http://dx.doi.org.proxy.lib.uwaterloo.ca/10.1007/BF01195123},
      review={\MR{1627901 (2000g:47016)}},
}

\bib{dritschel2005}{article}{
      author={Dritschel, Michael~A.},
      author={McCullough, Scott~A.},
       title={Boundary representations for families of representations of
  operator algebras and spaces},
        date={2005},
        ISSN={0379-4024},
     journal={J. Operator Theory},
      volume={53},
      number={1},
       pages={159\ndash 167},
      review={\MR{2132691 (2006a:47095)}},
}

\bib{effrosruan1990}{article}{
      author={Effros, Edward~G.},
      author={Ruan, Zhong-Jin},
       title={On nonselfadjoint operator algebras},
        date={1990},
        ISSN={0002-9939},
     journal={Proc. Amer. Math. Soc.},
      volume={110},
      number={4},
       pages={915\ndash 922},
         url={http://dx.doi.org.uml.idm.oclc.org/10.2307/2047737},
      review={\MR{986648}},
}

\bib{EL1992}{article}{
      author={Exel, Ruy},
      author={Loring, Terry~A.},
       title={Finite-dimensional representations of free product
  {$C^*$}-algebras},
        date={1992},
        ISSN={0129-167X},
     journal={Internat. J. Math.},
      volume={3},
      number={4},
       pages={469\ndash 476},
         url={https://doi-org.uml.idm.oclc.org/10.1142/S0129167X92000217},
      review={\MR{1168356}},
}

\bib{GM1990}{article}{
      author={Goodearl, K.~R.},
      author={Menal, P.},
       title={Free and residually finite-dimensional {$C^*$}-algebras},
        date={1990},
        ISSN={0022-1236},
     journal={J. Funct. Anal.},
      volume={90},
      number={2},
       pages={391\ndash 410},
         url={https://doi-org.uml.idm.oclc.org/10.1016/0022-1236(90)90089-4},
      review={\MR{1052340}},
}

\bib{hadwin2014}{article}{
      author={Hadwin, Don},
       title={A lifting characterization of {RFD} {$\rm C^*$}-algebras},
        date={2014},
        ISSN={0025-5521},
     journal={Math. Scand.},
      volume={115},
      number={1},
       pages={85\ndash 95},
         url={https://doi-org.uml.idm.oclc.org/10.7146/math.scand.a-18004},
      review={\MR{3250050}},
}

\bib{hamana1979}{article}{
      author={Hamana, Masamichi},
       title={Injective envelopes of operator systems},
        date={1979},
        ISSN={0034-5318},
     journal={Publ. Res. Inst. Math. Sci.},
      volume={15},
      number={3},
       pages={773\ndash 785},
  url={http://dx.doi.org.proxy.lib.uwaterloo.ca/10.2977/prims/1195187876},
      review={\MR{566081 (81h:46071)}},
}

\bib{hopenwasser1973}{article}{
      author={Hopenwasser, Alan},
       title={Boundary representations on {$C^{\ast} $}-algebras with matrix
  units},
        date={1973},
        ISSN={0002-9947},
     journal={Trans. Amer. Math. Soc.},
      volume={177},
       pages={483\ndash 490},
         url={https://doi-org.uml.idm.oclc.org/10.2307/1996610},
      review={\MR{0322522}},
}

\bib{KRmem}{article}{
      author={Katsoulis, Elias},
      author={Ramsey, Christopher},
       title={Crossed products of operator algebras},
        date={2017},
     journal={to appear Mem. Amer. Math. Soc.},
}

\bib{meyer2001}{article}{
      author={Meyer, Ralf},
       title={Adjoining a unit to an operator algebra},
        date={2001},
        ISSN={0379-4024},
     journal={J. Operator Theory},
      volume={46},
      number={2},
       pages={281\ndash 288},
      review={\MR{1870408}},
}

\bib{mittal2010}{article}{
      author={Mittal, Meghna},
      author={Paulsen, Vern~I.},
       title={Operator algebras of functions},
        date={2010},
        ISSN={0022-1236},
     journal={J. Funct. Anal.},
      volume={258},
      number={9},
       pages={3195\ndash 3225},
  url={http://dx.doi.org.proxy.lib.uwaterloo.ca/10.1016/j.jfa.2010.01.006},
      review={\MR{2595740 (2011h:47150)}},
}

\bib{MS1998}{article}{
      author={Muhly, Paul~S.},
      author={Solel, Baruch},
       title={Tensor algebras over {$C^*$}-correspondences: representations,
  dilations, and {$C^*$}-envelopes},
        date={1998},
        ISSN={0022-1236},
     journal={J. Funct. Anal.},
      volume={158},
      number={2},
       pages={389\ndash 457},
         url={https://doi-org.uml.idm.oclc.org/10.1006/jfan.1998.3294},
      review={\MR{1648483}},
}

\bib{paulsen2002}{book}{
      author={Paulsen, Vern},
       title={Completely bounded maps and operator algebras},
      series={Cambridge Studies in Advanced Mathematics},
   publisher={Cambridge University Press},
     address={Cambridge},
        date={2002},
      volume={78},
        ISBN={0-521-81669-6},
      review={\MR{1976867 (2004c:46118)}},
}

\bib{paulsen1984PAMS}{article}{
      author={Paulsen, Vern~I.},
       title={Completely bounded homomorphisms of operator algebras},
        date={1984},
        ISSN={0002-9939},
     journal={Proc. Amer. Math. Soc.},
      volume={92},
      number={2},
       pages={225\ndash 228},
         url={http://dx.doi.org.proxy.lib.uwaterloo.ca/10.2307/2045190},
      review={\MR{754708 (85m:47049)}},
}

\bib{paulsen1984}{article}{
      author={Paulsen, Vern~I.},
       title={Every completely polynomially bounded operator is similar to a
  contraction},
        date={1984},
        ISSN={0022-1236},
     journal={J. Funct. Anal.},
      volume={55},
      number={1},
       pages={1\ndash 17},
         url={http://dx.doi.org/10.1016/0022-1236(84)90014-4},
      review={\MR{733029 (86c:47021)}},
}

\bib{paulsen1992}{article}{
      author={Paulsen, Vern~I.},
       title={Representations of function algebras, abstract operator spaces,
  and {B}anach space geometry},
        date={1992},
        ISSN={0022-1236},
     journal={J. Funct. Anal.},
      volume={109},
      number={1},
       pages={113\ndash 129},
         url={https://doi-org.uml.idm.oclc.org/10.1016/0022-1236(92)90014-A},
      review={\MR{1183607}},
}

\bib{pestov1994}{article}{
      author={Pestov, Vladimir~G.},
       title={Operator spaces and residually finite-dimensional
  {$C^*$}-algebras},
        date={1994},
        ISSN={0022-1236},
     journal={J. Funct. Anal.},
      volume={123},
      number={2},
       pages={308\ndash 317},
         url={https://doi-org.uml.idm.oclc.org/10.1006/jfan.1994.1090},
      review={\MR{1283030}},
}

\bib{phelps2001}{book}{
      author={Phelps, Robert~R.},
       title={Lectures on {C}hoquet's theorem},
     edition={Second Ed.},
      series={Lecture Notes in Mathematics},
   publisher={Springer-Verlag, Berlin},
        date={2001},
      volume={1757},
        ISBN={3-540-41834-2},
         url={http://dx.doi.org.uml.idm.oclc.org/10.1007/b76887},
      review={\MR{1835574}},
}

\bib{pimsner1997}{incollection}{
      author={Pimsner, Michael~V.},
       title={A class of {$C^*$}-algebras generalizing both {C}untz-{K}rieger
  algebras and crossed products by {${\bf Z}$}},
        date={1997},
   booktitle={Free probability theory ({W}aterloo, {ON}, 1995)},
      series={Fields Inst. Commun.},
      volume={12},
   publisher={Amer. Math. Soc., Providence, RI},
       pages={189\ndash 212},
      review={\MR{1426840}},
}

\bib{popescu1991}{article}{
      author={Popescu, Gelu},
       title={von {N}eumann inequality for {$(B({\scr H})^n)_1$}},
        date={1991},
        ISSN={0025-5521},
     journal={Math. Scand.},
      volume={68},
      number={2},
       pages={292\ndash 304},
      review={\MR{1129595 (92k:47073)}},
}

\bib{popescu1996}{article}{
      author={Popescu, Gelu},
       title={Non-commutative disc algebras and their representations},
        date={1996},
        ISSN={0002-9939},
     journal={Proc. Amer. Math. Soc.},
      volume={124},
      number={7},
       pages={2137\ndash 2148},
         url={https://doi-org.uml.idm.oclc.org/10.1090/S0002-9939-96-03514-9},
      review={\MR{1343719}},
}

\bib{TWW2017}{article}{
      author={Tikuisis, Aaron},
      author={White, Stuart},
      author={Winter, Wilhelm},
       title={Quasidiagonality of nuclear {$C^\ast$}-algebras},
        date={2017},
        ISSN={0003-486X},
     journal={Ann. of Math. (2)},
      volume={185},
      number={1},
       pages={229\ndash 284},
         url={https://doi-org.uml.idm.oclc.org/10.4007/annals.2017.185.1.4},
      review={\MR{3583354}},
}

\end{biblist}
\end{bibdiv}
\bibliographystyle{plain}


\end{document}